\documentclass{article}
\usepackage[a4paper]{geometry}

\usepackage{amsthm,amsmath,mathtools,bm,enumerate,subcaption, amssymb,dsfont}

\usepackage[utf8]{inputenc} 
\usepackage[T1]{fontenc}    
\usepackage{hyperref}       
\usepackage{url}            
\usepackage{booktabs}       
\usepackage{amsfonts}       
\usepackage{nicefrac}       
\usepackage{microtype,color}      
\usepackage{graphicx}
\usepackage{standalone}
\usepackage{algorithm}
\usepackage{tikz}
\usetikzlibrary{patterns}
\usetikzlibrary{shadows.blur}
\usetikzlibrary{calc}
\usetikzlibrary{matrix}
\usetikzlibrary{calc}

\newcommand\Ret{\mathrm{Ret}}
\newcommand\adm{\mathrm{adm}}

\usepackage[mathcal]{eucal}

\newcommand{\RR}{\mathbb{R}}
\newcommand{\NN}{\mathbb{N}}

\newcommand{\TT}{\mathbb{T}}
\renewcommand{\SS}{\mathbb{S}}

\newcommand{\Cc}{\mathcal{C}}

\newcommand{\Ff}{\mathcal{F}}

\newcommand{\Hh}{\mathcal{H}}

\newcommand{\Mm}{\mathcal{M}}

\newcommand{\Pp}{\mathcal{P}}
\newcommand{\Qq}{\mathcal{Qq}}

\newcommand{\Tt}{\mathcal{T}}

\newcommand{\Xx}{\mathcal{X}}

\renewcommand{\d}{\mathrm{d}}

\newcommand{\id}{\mathrm{id}}

\newcommand{\h}{\mathsf{h}}

\renewcommand{\div}{\mathrm{div}}
\newcommand{\BL}{\mathrm{BL}}
\newcommand{\Lip}{\mathrm{Lip}}

\DeclareMathOperator{\tr}{tr}
\DeclareMathOperator{\diag}{diag}

\DeclareMathOperator{\vol}{vol}

\DeclareMathOperator{\dist}{dist}
\DeclareMathOperator{\spt}{spt}

\newtheorem{theorem}{Theorem}[section]
\newtheorem{proposition}[theorem]{Proposition}
\newtheorem{lemma}[theorem]{Lemma}
\newtheorem{corollary}[theorem]{Corollary}

\newtheorem{definition}[theorem]{Definition}

\usepackage{animate}

\title{Sparse Optimization on Measures \\ with Over-parameterized Gradient Descent}
\author{
  L\'ena\"ic Chizat\thanks{CNRS, Laboratoire de Math\'ematiques d'Orsay, Universit\'e Paris-Saclay, 91405, Orsay, France.}
}

\begin{document}

\maketitle

\begin{abstract}
Minimizing a convex function of a measure with a sparsity-inducing penalty is a typical problem arising, e.g., in sparse spikes deconvolution or two-layer neural networks training. We show that this problem can be solved by discretizing the measure and running non-convex gradient descent on the positions and weights of the particles. For measures on a $d$-dimensional manifold and under some non-degeneracy assumptions, this leads to a global optimization algorithm with a complexity scaling as $\log(1/\epsilon)$ in the desired accuracy $\epsilon$, instead of $\epsilon^{-d}$ for convex methods. The key theoretical tools are a local convergence analysis in Wasserstein space and an analysis of a perturbed mirror descent in the space of measures. Our bounds involve quantities that are exponential in $d$ which is unavoidable under our assumptions.
\end{abstract}

\section{Introduction}\label{sec:intro}
Finding parsimonious descriptions of complex observations is an important problem in machine learning and signal processing.  In its simplest form, this task boils down to searching for an element in a Hilbert space $\Ff$ that is close to a certain $f_0\in \Ff$ --- the observations --- and that is a linear combination of a few elements from a parameterized set $\{\phi(\theta)\}_{\theta \in \Theta}\subset \Ff$ --- the parsimonious description. This can be formulated as a minimization problem where the linear combination is expressed through an unknown measure $\nu$ and the distance to $f_0$ is quantified using a smooth convex loss function $R:\Ff\to \RR$, such as the square loss $R(f)=\frac12 \Vert f-f_0\Vert^2_\Ff$.  The problem to solve is then
\begin{equation}\label{eq:objective}
J^*\coloneqq \min_{\nu \in \Mm_+(\Theta)} J(\nu), \qquad \text{}\qquad J(\nu) :=R\left( \int_\Theta \phi(\theta) \d \nu(\theta)\right) + \lambda  \nu (\Theta)
\end{equation}
where $\Mm_+(\Theta)$ is the set of nonnegative measures $\nu$ on the parameter space $\Theta$ with finite total mass $\nu(\Theta)<\infty$ and $\lambda> 0$ is the regularization strength. This formulation also covers minimization over signed measures with total variation regularization, by replacing $\Theta$ with the disjoint union of two copies of $\Theta$ where $\phi$ takes opposite values, see Appendix~\ref{app:signed}. A large body of research has exhibited the favorable properties of minimizers of such problems~\cite{bach2017breaking, de2012exact, mairal2014sparse} with a statistical or variational viewpoint, showing in particular that $\lambda$ favors sparser solutions and increases stability as it gets larger, at the expense of introducing a stronger bias. The present paper deals with the optimization aspect: our goal is to design algorithms that return $\epsilon$-accurate solutions with a guaranteed computational complexity. When the set $\Theta$ is a finite set, this is a finite dimensional convex optimization problem that is well understood~\cite{boyd2004convex, bach2012optimization}. However, convex approaches are generally inefficient when $\Theta$ is a continuous space, such as a $d$-dimensional manifold, where the need to discretize the space leads to a complexity scaling as $\epsilon^{-d}$ in the accuracy $\epsilon$. We consider the following setting:
\begin{enumerate}
\item[] {\sf (A1)}  $\Theta$ is a compact $d$-dimensional Riemannian manifold without boundaries. The functions $\phi:\Theta\to \Ff$ and $R:\Ff\to \RR_+$ are twice Fr\'echet differentiable, with locally Lipschitz second-order derivatives, and $\nabla R$ is bounded on sublevel sets.
\end{enumerate}
The algorithm that we analyze in this paper is simple to describe: 
\emph{initialize with a discrete measure and run gradient descent on the positions and weights of the particles}. We will see that when the problem~\eqref{eq:objective} admits sparse solutions and is non-degenerate, this over-parameterized non-convex gradient descent has a complexity scaling as $\log(1/\epsilon)$ in the accuracy $\epsilon$. We make the following contributions:
\begin{itemize}
\item[--] In Section~\ref{sec:presentation}, we introduce the \emph{conic particle gradient descent} algorithm to solve optimization problems in the space of measures and discuss several of its interpretations.
\item[--] In Section~\ref{sec:local}, we show under under certain non-degeneracy assumptions that there is a sublevel of $J$ starting from which this algorithm converges exponentially fast to minimizers.
\item[--] In Section~\ref{sec:global}, we show that for suitable choices of gradient and initialization, this algorithm converges to global minimizers. The proof combines the  result of Section~\ref{sec:local} with an analysis of a perturbed mirror descent in the space of measures. The number of iterations required to reach an accuracy $\epsilon$ is polynomial in the characteristics of the problem and logarithmic in~$\epsilon$. In contrast, the required number of particles depends exponentially on the dimension $d$, which is unavoidable under our assumptions. 
\item[--] We report results of numerical experiments in Section~\ref{sec:numerics}, where the various insights brought by our analysis about local and global behaviors are investigated.
\end{itemize}

\subsection{Examples of applications}\label{sec:applications} 
\begin{figure}
\centering
\begin{subfigure}{0.3\linewidth}
\centering
\includegraphics[scale=0.5]{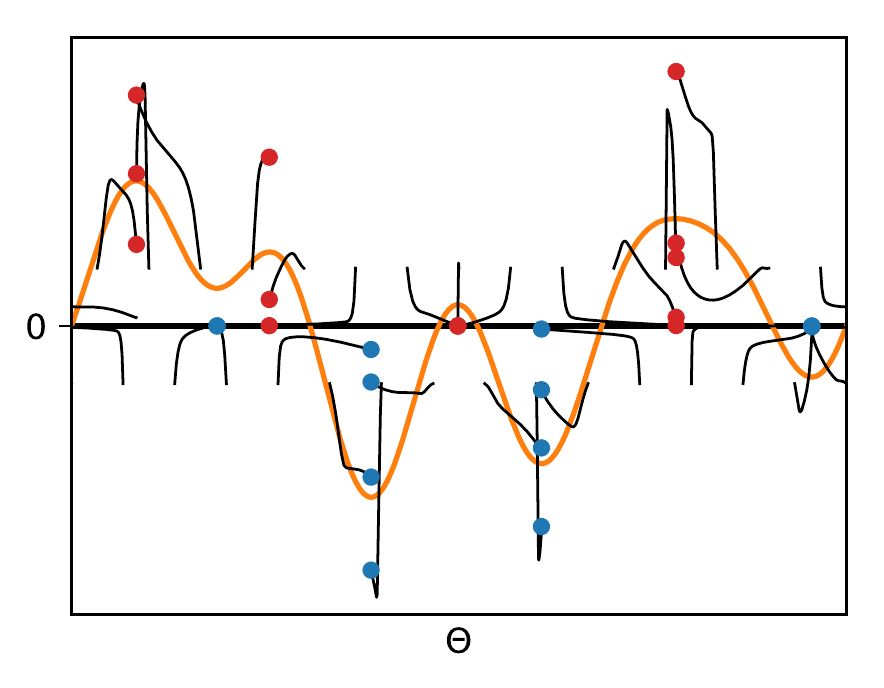}
\caption{Sparse deconvolution}
\end{subfigure}%
\begin{subfigure}{0.3\linewidth}
\centering
\includegraphics[scale=0.5]{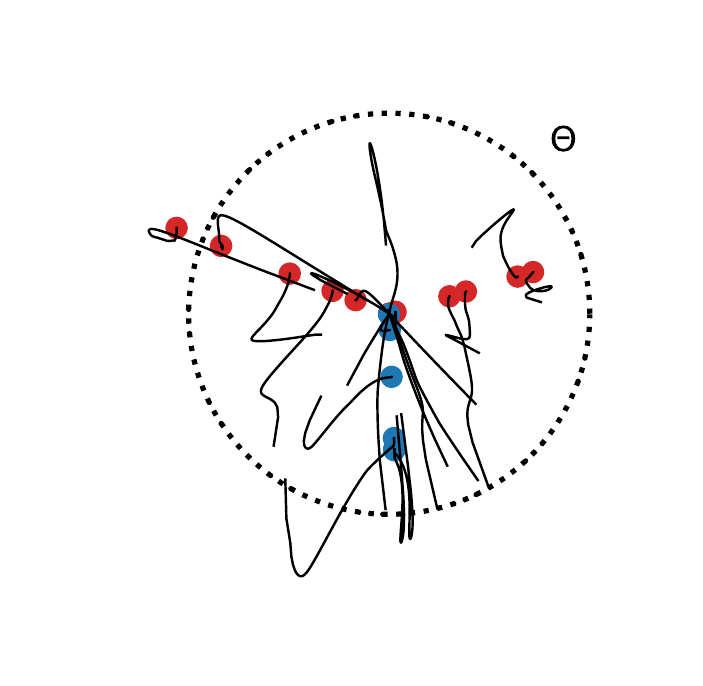}
\caption{Two-layer ReLU neural net}
\end{subfigure}%
\begin{subfigure}{0.3\linewidth}
\centering
\includegraphics[scale=0.5]{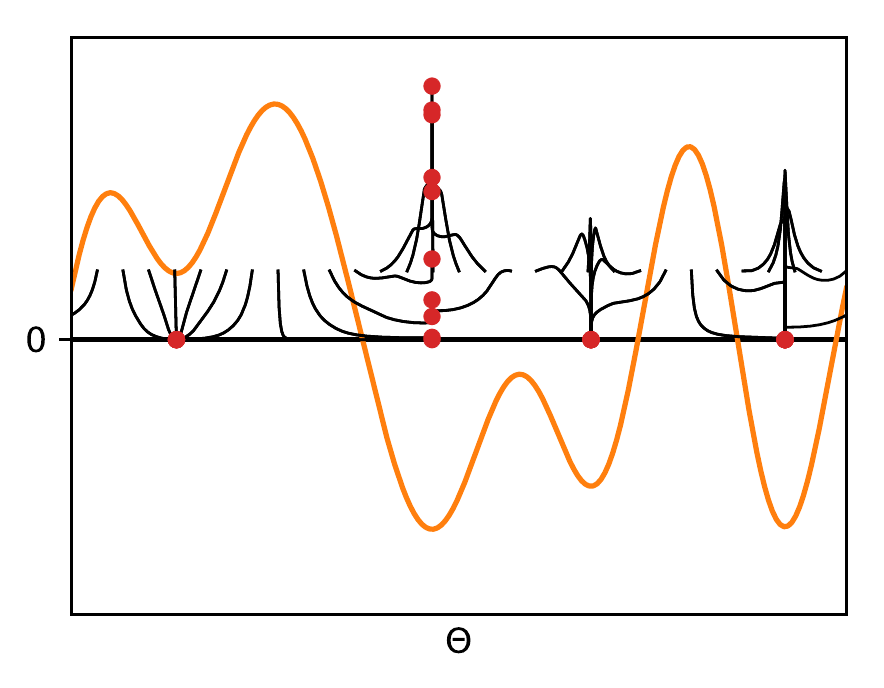}
\caption{Generic optimization}
\end{subfigure}
\caption{Three examples of conic particle gradient descents (Algorithm~\ref{alg:gradientbased}). Trajectories of particles in black and their limits in red (positive mass) or blue (negative mass). }
\label{fig:cover}
\end{figure}
As the problem of finding the simplest linear decomposition over a continuous dictionary is a very natural one, problems of the form~\eqref{eq:objective} appear in a large variety of situations, see~\cite{boyd2017alternating} for an extensive list. In this paper, our numerical illustrations are focused on two applications, chosen for their practical importance and also because they illustrate the variety of behaviors that can be encountered. We also mention a third example to emphasize on the extreme generality --- and thus the intrinsic limits --- of our analysis. These three cases are illustrated on Figure~\ref{fig:cover}.

\paragraph{Sparse deconvolution.} In this application, we want to recover a signal that consists of a mixture of spikes/impulses on $\Theta$ given a noisy and filtered observation $f_0$ in the space $\Ff = L^2(\Theta)$ of square-integrable real-valued functions on $\Theta$. When one defines $\phi(\theta): x \mapsto \psi( x -\theta)$ the translations of the filter impulse response $\psi$ and $R$ the squared loss, solving~\eqref{eq:objective} allows to reconstruct the mixture of impulses with some guarantees, see e.g.~\cite{duval2015exact,de2012exact,poon2018geometry}. In this  typically low dimensional application, solving~\eqref{eq:objective} to a high accuracy is crucial. Both the signed and nonnegative case have practical motivations (see Appendix~\ref{app:signed} for how to handle the signed case). Figure~\ref{fig:cover}-(a) illustrates the behavior of particle gradient descent for the signed case on the $1$-torus, where the observed signal is shown in orange. Figure~\ref{fig:deconv2D} illustrates the unsigned case on the $2$-torus. 

\paragraph{Two-layer neural networks.} Here the goal is to select, within a specific class, a function that maps features in $\RR^{d-1}$ to labels in $\RR$ from the observation of a joint distribution of features and labels. This corresponds to $\Ff$ being the space of real-valued functions on $\RR^{d-1}$ which are square-integrable under the distribution of features, $R$ being e.g., the quadratic or the logistic loss function, and $\phi(\theta) : x\mapsto \sigma (\sum_{i=1}^{d-1} \theta_i x_i + \theta_d)$ with an activation function $\sigma : \RR \to \RR$. Common choices are the sigmoid function or the rectified linear unit~\cite{haykin1994neural,goodfellow2016deep}. In this application, $d$ is typically large and it is not clear yet how to verify the non-degeneracy assumptions a priori, so our global convergence bounds are not useful. Still, the local analysis in Section~\ref{sec:local} gives insights on the local behavior in the over-parameterized regularized setting and explains well the behavior observed in numerical experiments. With the ReLU activation, the method we analyze boils down to the classical gradient descent algorithm, see the remark in Section~\ref{sec:conic} about the $2$-homogeneous case. Figure~\ref{fig:cover}-(b)  illustrates this case, by plotting the trajectories of $\vert a_i\vert\cdot b_i\in \RR^{2}$ where $a_i\in \RR$ is the output weight of neuron $i$ and $b_i\in \RR^2$ its hidden weights (the color represents the sign of $a_i$).

\paragraph{Non-convex optimization.} Lastly, the minimization of any smooth function on a manifold $\phi:\Theta\to \RR$ is covered by~\eqref{eq:objective}, as proved in Appendix~\ref{app:genericnonconvex}. For this problem, our algorithm is analogous to running independently several gradient-based minimization with diverse initializations, because the various particles simply follow the gradient field of $\phi$ and only interact through their masses. This case is illustrated on Figure~\ref{fig:cover}-(c) where the function to minimize (here on the $1$-torus) is plotted in orange.  We recover the standard fact that random search as to be complemented with local search if one wants complexity that is reasonable in the precision. We stress that this is not the situation that motivates our analysis. Instead, we are interested in the case of general interactions between the particles, which is when we obtain novel insights.

\begin{figure}
\centering
\begin{subfigure}{0.33\linewidth}
\centering
\includegraphics[scale=0.6,trim=0.5cm 0.7cm 0cm 0.2cm, clip]{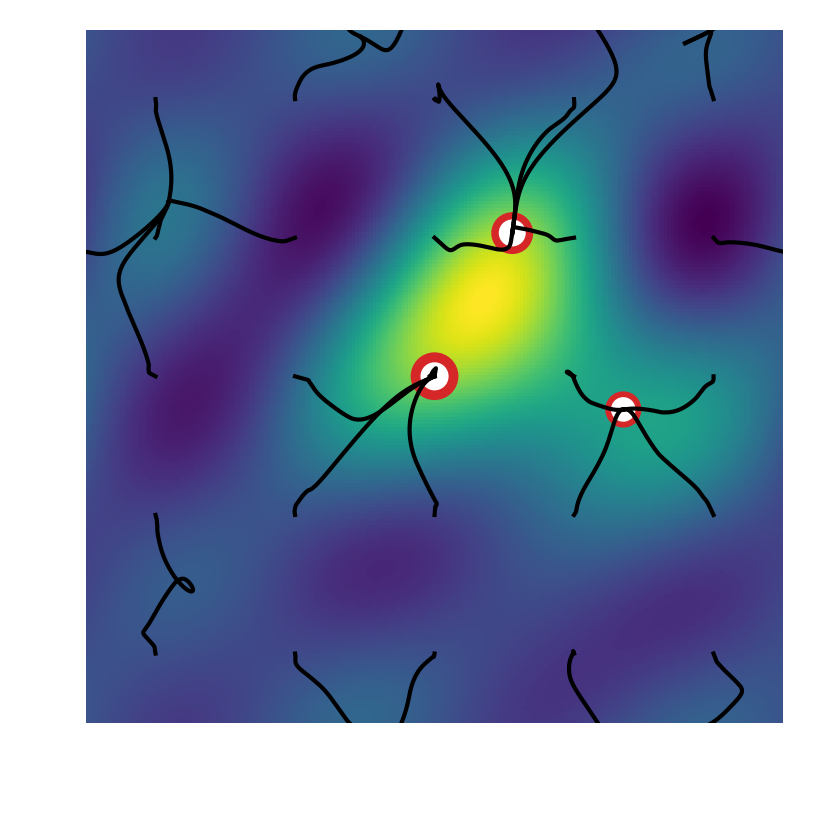}
\caption{$m=25$ particles}
\end{subfigure}%
\begin{subfigure}{0.33\linewidth}
\centering
\includegraphics[scale=0.6,trim=0.5cm 0.7cm 0cm 0.2cm, clip]{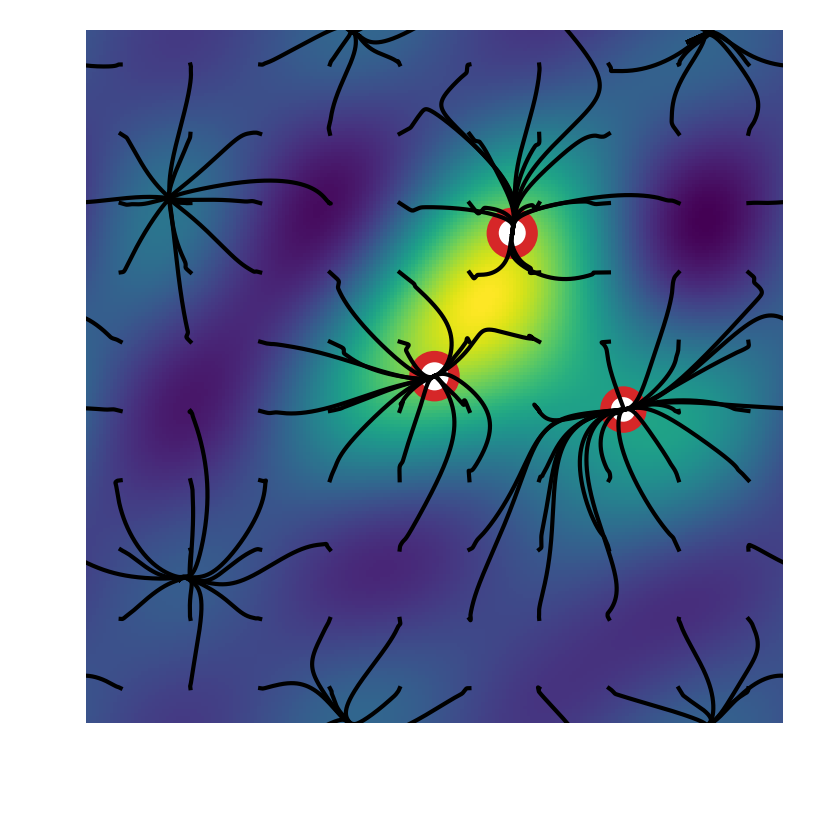}
\caption{$m=100$ particles}
\end{subfigure}%
\caption{Illustration of 2D sparse deconvolution with conic particle gradient descent (Algorithm~\ref{alg:gradientbased}). Ground truth spikes in white, spatial trajectories in black and final location/mass in red.}
\label{fig:deconv2D}
\end{figure}

\subsection{Related work}\label{sec:related}

\paragraph{Sparse optimization on measures.}  
Problems with the structure~\eqref{eq:objective} have a long history in optimization when $\Theta$ is discrete, and is typically solved with ISTA~\cite{daubechies2004iterative}, mirror descent~\cite{nemirovsky1983problem,beck2003mirror} or variants of those algorithms. When $\Theta$ is continuous, the one dimensional case can sometimes be dealt with specific algorithms~\cite{candes2014towards,champagnat2019atom}. In higher dimensions, the classical algorithms are conditional gradient algorithms (also known as Franck-Wolfe)~\cite{bredies2013inverse,denoyelle2018sliding,boyd2017alternating}, moment methods~\cite{de2017exact,catala2017low,dumitrescu2007positive} and adaptive sampling/exchange algorithms~\cite{flinth2017exact,flinth2019linear}. Often, these algorithms are complemented with non-convex updates on the particle positions, which considerably improves their behavior. Given an initial condition that is close to the optimum and with the same structure (i.e.\ without over-parameterization), the local convergence for non-convex gradient descent is studied in~\cite{traonmilin2018basins, flinth2019linear}. 

\paragraph{Wasserstein gradient flows for optimization.}  The dynamics of two-layer neural networks optimization when the number of hidden units  grows unbounded is studied in~\cite{nitanda2017stochastic, chizat2018global,mei2018mean,rotskoff2018parameters,sirignano2019mean}. This series of work has led to various insights related to stochastic fluctuations and global convergence. The present paper can be seen as a quantitative counterpart to~\cite{chizat2018global}, although we consider a more restrictive setting\footnote{The algorithm we study in this paper corresponds to the ``$2$-homogeneous case'' in~\cite{chizat2018global}. Also,~\cite{chizat2018global} allows non-smooth regularizers and does not require non-degeneracy.}. A global rate of convergence is obtained in~\cite{wei2019regularization} but for a modified dynamic where particles are re-sampled at each iteration. Instead, we focus on the basic case where particles are  only sampled once at the beginning of the algorithm. It should be mentioned that our analysis is different from the line of research on lazy over-parameterized models~\cite{chizat2019lazy}  initiated by~\cite{du2018gradient,jacot2018neural}, which does not apply to the regularized case and to the unsigned case. Finally, in the parametric case where the unknown measure is assumed to belong to a finite dimensional probability model, Wasserstein natural gradient~\cite{amari1998natural, li2018natural, chen2018wasserstein} or accelerated versions~\cite{wang2019accelerated} have been proposed. Our analysis is however of non-parametric nature because the number of parameters is not fixed a priori in the analysis.

\paragraph{Related techniques.} Our framework involves the theory of optimization on manifolds~\cite{absil2009optimization} and of Wasserstein gradient flows~\cite{ambrosio2008gradient}. Some inspiration and interpretations of the algorithm under consideration come from unbalanced optimal transport theory~\cite{liero2018optimal, kondratyev2016new, chizat2018interpolating} and in particular, from the lifting construction in~\cite{liero2018optimal}. Finally, our local analysis includes a functional and a gradient \L{}ojasiewicz inequality of order $2$ in Wasserstein space. Such inequalities were studied in~\cite{hauer2019kurdyka,blanchet2018family} for displacement convex functions, which does not cover our setting.

\subsection{Notation}
The set of signed (resp.~nonnegative) finite Borel measures on a metric space $(\Xx,\dist)$ is denoted by $\Mm(\Xx)$ (resp. $\Mm_+(\Xx)$). The relative entropy, a.k.a.~Kullback-Leibler divergence, is defined for $\nu_1,\nu_2 \in \Mm_+(\Xx)$ as
$\Hh (\nu_1,\nu_2) = \int_{\Xx} \log(\d\nu_1/\d\nu_2) \d\nu_1 - \nu_1(\Xx)+\nu_2(\Xx)$ if $\nu_1$ is absolutely continuous w.r.t.~$\nu_2$, and $+\infty$ otherwise. The $p$-Wasserstein distance on the set $\Pp_p(\Xx)$ of probability measures with finite $p$-th moment is defined, for $\mu_1,\mu_2 \in \Pp_p(\Xx)$ as 
\[
W_p(\mu_1,\mu_2) = \left(\min_{\gamma \in \Pi(\mu_1,\mu_2)} \int \dist(x_1,x_2)^p\d\gamma(x_1,x_2)\right)^{1/p}
\]
where $\Pi(\mu_1,\mu_2)$ is the set of measures on $\Xx \times \Xx$ with marginals $\mu_1$ and $\mu_2$. The distance $W_\infty$ between compactly supported probabilities is defined as the limit of $W_p$ as $p\to \infty$ and can be directly defined as $W_\infty(\mu_1,\mu_2) = \inf_{\gamma \in \Pi(\mu_1,\mu_2)} \max_{(x_1,x_2)\in \spt \gamma} \dist(x_1,x_2)$~\cite{santambrogio2015optimal}. We also define the Bounded-Lipschitz norm for a continuous function $\psi:\Xx\to \RR$ as $\Vert \psi \Vert_{\BL} = \Vert \psi\Vert_\infty+\Lip(\psi)$ where $\Lip(\psi)$ is the Lipschitz constant of $\psi$ and its dual norm on $\Mm(\Xx)$ as 
$\Vert \nu \Vert^*_{\BL} \coloneqq \sup_{\Vert \varphi \Vert_{\BL}\leq 1} \int \varphi \, \d\nu$. For a Riemannian manifold $\Theta$, we denote by $T_\theta\Theta$ the tangent space of $\Theta$ at $\theta$ and by $\langle \cdot,\cdot\rangle_\theta : (T_\theta \Theta)^2 \to \RR_+$ the metric at $\theta$.

\section{Particle gradient descent}\label{sec:presentation}
\subsection{General case}
Consider a smooth increasing bijection $h:\RR_+\to \RR_+$ (such as a power function $r\mapsto r^p$) and a number of particles $m \in \NN^*$. The idea behind particle gradient-based algorithms is to parameterize the unknown measure $\nu$ as $\frac1m \sum_{i=1}^m h(r_i) \delta_{\theta_i}$ and to perform gradient-based optimization on the corresponding objective
\begin{equation}\label{eq:objectivediscrete}
F_{m}((r_1,\theta_1),\dots,(r_{m},\theta_m)) \coloneqq  R\left(\frac1{m} \sum_{i=1}^{m} h(r_i)\phi(\theta_i)\right) + \frac\lambda{m} \sum_{i=1}^{m} h(r_i),
\end{equation}
where the parameters $(r_i,\theta_i)$ of each particle belong to $\Omega \coloneqq \RR_+ \times \Theta$ endowed with a specific choice of metric. Clearly, if $J$ admits a minimizer that is a mixture of $m^\star$ atoms with $m^\star\leq m$, then it is sufficient to minimize $F_m$ from Eq.~\ref{eq:objectivediscrete} for solving \eqref{eq:objective}. While~\eqref{eq:objectivediscrete} is finite dimensional, it is typically non-convex with possibly some strict local minima. Still, when $R$ is convex and for $h(r)=r^p$ for $p\in \{1,2\}$, the message from~\cite{chizat2018global}  (see Theorem~\ref{th:consistency}) is that  solving~\eqref{eq:objectivediscrete} to global optimality with first-order methods is still possible by using over-parameterization, i.e.\ choosing $m$ much larger than $m^\star$.  Such a method involve various key hyper-parameters which role is discussed throughout the paper. They include (i) the choice of the function $h$ (ii) the choice of the metric on $\Omega^m$ and (iii) the choice of the initialization.

\paragraph{Expression of the gradient.} Under~{\sf(A1)}, the objective $J$, seen as a function on the space $\Mm(\Theta)$ endowed with the total variation norm, is Fr\'echet-differentiable. Its differential at $\nu\in \Mm(\Theta)$ can be represented by the function $J'_\nu : \Theta \to \RR$ given by
\begin{equation}\label{eq:firstvariation}
J'_\nu(\theta) = \left\langle \phi(\theta),\nabla R\left(\int_\Theta \phi(\theta) \d\nu(\theta)\right)\right\rangle_{\Ff} +\lambda,
 \end{equation}
 in the sense that for any $\sigma \in\Mm(\Theta)$, it holds
$\frac{d}{d\epsilon} J(\nu + \epsilon \sigma)\vert_{\epsilon=0} = \int_{\Theta} J'_\nu(\theta) \d\sigma(\theta)$. Now, consider a metric on $(\Omega^*)^m$ that is the average $(1/m)\sum_{i=1}^m \langle \cdot, \cdot \rangle_{(r_i,\theta_i)}$ of metrics on each factor $\Omega^*\coloneqq \RR_+^*\times \Theta$, where $\RR_+^*$ is the set of positive real numbers, of the form
\begin{equation}\label{eq:metric}
\langle(\delta r_1,\delta \theta_1),(\delta r_2,\delta \theta_2) \rangle_{(r,\theta)} = \alpha(r)^{-1} \delta r_1\delta r_2 +\beta(r)^{-1}\langle \delta \theta_1,\delta\theta_2\rangle_{\theta} 
\end{equation}
where $\alpha$ and $\beta$ are smooth functions $\RR_+^*\to \RR_+^*$ to be specified\footnote{Extension of the metric and gradients to the whole of $\Omega$ can be made on a case by case basis, see Section~\ref{sec:conic}.}, $ (r,\theta) \in \Omega^*$, $\delta r_1, \delta r_2 \in \RR$ and $\delta \theta_1, \delta \theta_2 \in T_\theta\Theta$. Using the fact that gradients are characterized by the relation $\d F_m(x)(\delta x) = \langle \nabla F_m(x),\delta x\rangle$, we get that the gradient of $F_m$ is given, in components, by 
\begin{align}\label{eq:gradientgeneric}
\left\{
\begin{aligned}
\nabla_{r_i} F_m((r_i,\theta_i)_{i=1}^m) &= \alpha(r_i)h'(r_i) J'_{\nu}(\theta_i)\\
\nabla_{\theta_i} F_m ((r_i,\theta_i)_{i=1}^m) &= \beta(r_i) h(r_i) \nabla J'_{\nu}(\theta_i)
\end{aligned}
\right. && \text{where} && \nu = \frac1m \sum_{i=1}^m h(r_i)\delta_{\theta_i}.
\end{align}

\paragraph{Lifted problem in Wasserstein space.} Assume now that $h$ has at most quadratic growth, and that the metric is defined on the whole of $\Omega$. One can then see the discrete problem~\eqref{eq:objectivediscrete} as a discretization of a problem on the space $\Pp_2(\Omega)$ of probability measures on $\Omega$ with finite second moment endowed with the Wasserstein-$2$ metric given by
\begin{equation}\label{eq:liftedproblem}
F^\star = \min_{\mu \in \Pp_2(\Omega)} F(\mu)\quad \text{where}\quad F(\mu) \coloneqq \left( \int_\Omega h(r) \phi(\theta) \d\mu(r,\theta) \right) +  \lambda \int_\Omega h(r) \d\mu(r,\theta).
\end{equation}
This point of view leads to insights on the properties of $F_m$ that are independent of $m$, which is crucial for our theoretical analysis. For a measure $\mu\in \Pp_2(\Omega)$, we define following~\cite{liero2018optimal} the \emph{homogeneous projection} operator $\h:\Pp_2(\Omega)\to \Mm_+(\Theta)$ where  $\h\mu$ is characterized by
\begin{align*}
\int_{\Theta} \varphi(\theta) \d(\h\mu)(\theta) = \int_{\Omega} h(r)\varphi(\theta)\d\mu(r,\theta)
\end{align*}
for any continuous function $\varphi:\Theta\to \RR$. With this operator, we simply have $F(\mu)=J(\h\mu)$. 

\paragraph{Gradient flow.} There are various ways to optimize~\eqref{eq:objectivediscrete} with first order methods. Instead of directly focusing on a specific method, we first consider the gradient flow of $F_m$, as it is known that (stochastic) gradient descent~\cite{gautschi1997numerical, kushner2003stochastic} approximates this dynamics. Let us call $x = (r_i,\theta_i)_{i=1}^m \in \Omega^m$ the variable of $F_m$. A gradient flow of $F_m$ is an absolutely continuous curve 
$( x(t))_{t\geq 0}$ in $\Omega^m$ that satisfies 
 \[
 x'(t) = -\nabla F_m(x(t))
 \]
  for $t\geq 0$, with the gradient given in Eq.~\eqref{eq:gradientgeneric}. Note that if $h'(r)\alpha(r)^{-1}$ does not tend to $0$ as $r\to 0$, then the non-negativity constraint on $r$ should be explicitly enforced, which requires the notion of subgradient flows, see~\cite{chizat2018global} for details in our setting.

\paragraph{Wasserstein gradient flow.} It is also possible to directly study the optimization dynamics in the space $\Pp_2(\Omega)$ for the functional $F$ of Eq.~\eqref{eq:liftedproblem}. For a measure $\nu \in \Mm_+(\Theta)$, consider the vector field on $\Omega$ with expression
\begin{equation*}
g_\nu(r,\theta) = \big(\alpha(r) h'(r) J'_{\nu}(\theta),\beta(r) h(r)\nabla J'_{\nu}(\theta)\big) \in \RR \times T_\theta \Theta.
\end{equation*}
We refer to $g_{\h \mu}$ as the Wasserstein gradient of $F$ at $\mu$ (this notation emphasizes that it only depends on $\mu$ through $\h \mu$). Gradient flows of $F_m$ are particular cases of \emph{Wasserstein gradient flows} of $F$. The latter are defined as the absolutely continuous curves $(\mu_t)_{t\geq 0}$ in $\Pp_2(\Omega)$ that satisfy
\begin{equation}\label{eq:WGF}
\partial_t \mu_t =  \div (\mu_t g_{\h\mu_t})
\end{equation}
in the weak sense, which means that for any differentiable function $\varphi: \Omega \to \RR$, it holds $\frac{\d}{\d t}\left(\int \varphi \,\d\mu_t \right) = - \int \nabla \varphi \cdot g_{\h\mu_t}\d\mu_t$, for almost every $t\geq 0$, see~\cite{santambrogio2015optimal}. This is a proper extension of the notion of gradient flow for $F_m$ in the sense that if $x(t)=(r_i(t),\theta_i(t))_{i=1}^m$ is a gradient flow of $F_m$ then it can be directly checked that  $t\mapsto \mu_t = \frac1m \sum_{i=1}^m \delta_{(r_i(t),\theta_i(t))}$ is a Wasserstein gradient flow of $F$.

\subsection{The conic case}\label{sec:conic}
As seen in Eq.~\eqref{eq:gradientgeneric}, the choice of the homogeneity degree and of the metric on $\Omega$ determine a specific way to combine the \emph{vertical}  and the \emph{spatial} components of the gradient (along the variable $r$ and $\theta$, respectively). From now on, we focus on what we refer to as the \emph{conic case}, which corresponds to the following assumption:
\begin{itemize}
\item [\sf (A2)] The mass parameterization is $h(r)=r^2$ and the metric on $\Omega^*$ is of the form Eq.~\eqref{eq:metric} with $(\alpha(r),\beta(r))=(\alpha,\beta/r^2)$ for some $\alpha,\beta>0$.
\end{itemize}
The corresponding geodesic distance is $\dist((r_1,\theta_2),(r_1,\theta_2))^2 = r_1^2 +r_2^2 -2r_1r_2\cos_{\pi}(\dist(\theta_1,\theta_2))$ where $\cos_\pi(z) = \cos(\min\{\pi,z\})$. This metric can be extended as a proper metric on $\widetilde \Omega$, defined as the set $\Omega$ where the subset $\{0\}\times \Theta$ is identified to a single point, known as the \emph{cone metric}, which is the canonical way to define a metric on $\tilde \Omega$~\cite{burago2001course}. In our context, identifying $\{0\}\times \Theta$ to a single point is desirable because a particle located in this set is a ``dead'' particle carrying no mass.

Plugging the metric into Eq.~\eqref{eq:gradientgeneric} gives the gradient (extended by continuity to $\{0\}\times \Theta$)
\begin{align*}
\left\{
\begin{aligned}
\nabla_{r_i} F_m((r_i,\theta_i)_{i=1}^m) &= 2\alpha r_i J'_{\nu}(\theta_i)\\
\nabla_{\theta_i} F_m ((r_i,\theta_i)_{i=1}^m) &=  \beta \nabla J'_{\nu}(\theta_i)
\end{aligned}
\right. && \text{where} && \nu = \frac1m \sum_{i=1}^m r_i^2\delta_{\theta_i},
\end{align*}
and the Wasserstein gradient is represented by the vector field
\begin{equation}\label{eq:wassersteingradient}
g_\nu(r,\theta) = \big(2 \alpha r J'_{\nu}(\theta), \beta \nabla J'_{\nu}(\theta)\big)  \in \RR \times T_\theta \Theta, \forall (r,\theta)\in \Omega.
\end{equation}
Existence of Wasserstein gradient flows under {\sf (A1-2)}, for any initialization in $\Pp_2(\Omega)$ can be proved along the same lines as in~\cite{chizat2018global}, see details in Appendix~\ref{app:existence}. Abstracting away its geometric derivation, the important aspects about our choice of gradient~\eqref{eq:wassersteingradient} are that its leads updates in $r$ which are multiplicative  and updates in $\theta$ which are independent of $r$. These two properties are crucial for our local convergence analysis  (Section~\ref{sec:local}). Moreover, multiplicative updates enjoy favorable convergence rates (Section~\ref{sec:global}).
The resulting structure and dynamics admits several interpretations.

\paragraph{Transport-growth interpretation.}
First, the projection $\nu_t = \h\mu_t$ of the gradient flow solves an advection-reaction equation. Importantly, this dynamics depends on $\mu_t$ only via the initialization $\h\mu_0$, which is a property specific to the conic setting.
\begin{proposition} \label{prop:projectedGF} Under {\sf (A1-2)}, let $(\mu_t)_{t\geq 0}$ be a Wasserstein gradient flow for $F$, with $\mu_0\in \Pp_2(\Omega)$. Then $\nu_t = \h\mu_t$ satisfies (in the weak sense)
\begin{equation}\label{eq:projectedGF}
\partial_t \nu_t = - 4 \alpha \nu_t J'_{\nu_t} + \beta \div (\nu_t \nabla J'_{\nu_t}).
\end{equation}
\end{proposition}
\begin{proof}
For any differentiable function $\varphi: \Theta\to \RR$, since $\mu_t$ is a Wasserstein gradient flow it holds
\begin{equation*}
\frac{\d}{\d t}\left(\int \varphi \,\d\nu_t \right) 
=- \int \langle \nabla (\h^*\varphi), g_{\h\mu_t}\rangle_{(r,\theta)} \d\mu_t = - \int (4 \alpha \varphi J'_{\nu_t} + \beta \nabla \varphi \cdot \nabla J'_{\nu_t})\d\nu_t,
\end{equation*}
which is the definition of weak solutions for~\eqref{eq:projectedGF}.
\end{proof}
When $\beta=0$, we recover the gradient flow of $J$ for the Fisher-Rao (or Hellinger) metric, which also corresponds to continuous time mirror descent on $\Mm_+(\Theta)$ for the entropy mirror map~\cite{krichene2015accelerated}. When $\alpha=0$, this is the gradient flow of $J$ for the Wasserstein metric~\cite{ambrosio2008gradient}. When $\alpha,\beta>0$, this is the gradient flow of the functional $J$ for the Wasserstein-Fisher-Rao metric, a.k.a.\ Hellinger-Kantorovich metric, see e.g.~\cite{gallouet2017jko}. Under Assumption~{\sf (A2)}, the dynamics~\eqref{eq:WGF} and~\eqref{eq:projectedGF} are directly related by Proposition~\ref{prop:projectedGF}. In the rest of this paper, we present the statements in terms of the projected dynamics $\nu_t$, although they also could be stated in terms of $\mu_t$. Note that an alternative discretization of the dynamic~\eqref{eq:projectedGF} was proposed in~\cite{rotskoff2019global} using particle birth-death.

\paragraph{Spherical coordinates interpretation.}
Consider the case when $\Theta=\SS^d$ is the $d$-dimensional sphere in $\RR^{d+1}$. Then, the space $\widetilde \Omega$ endowed with the cone metric and $\RR^{d+1}$ are isometric, through the spherical to Euclidean change of coordinates $(r,\theta)\mapsto r\theta$. Identifying $\widetilde \Omega$ with $\RR^{d+1}$ through this isometry, the class of functions of the form $r^p\phi(\theta)$ on $\widetilde\Omega$, for $p>0$ is simply the class of $p$-homogeneous functions on $\RR^{d+1}$. 

It follows that the conic setting we consider boils down, when $\Theta=\SS^d$ and $p=2$, to objectives defined on  $\Pp_2(\RR^{d+1})$ of the form
\begin{align}
F(\mu) = R\left( \int_{\RR^{d+1}} \psi(u)\d\mu(u) \right) + \lambda \int_{\RR^{d+1}} \Vert u\Vert^2_2 \d\mu(u)
\end{align}
with $\psi:\RR^{d+1}\to \Ff$ positively $2$-homogeneous. Moreover, the Wasserstein gradient on $\Pp_2(\widetilde\Omega)$ with the cone metric can be identified with the Wasserstein gradient on $\Pp_2(\RR^{d+1})$  with the Euclidean metric. One can thus understand our choice of conic metric and $p=2$ as a way to emulate the structure of $2$-homogeneous problems on $\RR^{d+1}$ in more general situations.

\paragraph{Asymptotic global convergence.} Let us recall the global convergence result of~\cite[Thm. 3.3]{chizat2018global}, in our setting and notations. We give in Appendix~\ref{app:existence} a simplified proof, enabled by our stronger smoothness assumptions. 
\begin{theorem}\label{th:consistency}
Under {\sf (A1-2)}, assume that $R$ is convex, that $\phi$ is $d$-times continuously differentiable, that $\nu_0\in \Mm_+(\Theta)$ has full support and that the projected gradient flow $(\nu_t)_{t\geq 0}$ converges weakly to some $\nu_{\infty} \in \Mm_+(\Theta)$. Then $\nu_{\infty}$ is a global minimizer of $J$. 
\end{theorem}
This theorem can be understood as a consistency result for conic particle gradient descent. It also raises several questions: under which conditions does $\nu_\infty$ exist? Can we guarantee a convergence rate ? Can we relax the full support condition on the initialization? In this paper, we answer positively to these questions in the particular case of non-degenerate sparse problems.

\subsection{Conic particle gradient descent algorithm}\label{sec:gradientdescent}

 \paragraph{Cone compatible retractions.}
 The definition of discrete gradient descent in a Riemannian setting requires to introduce the notion of retraction. In general, a retraction on a Riemannian manifold $\mathcal{M}$ with tangent bundle $T\mathcal{M}$ is a smooth map $\Ret:T\mathcal{M}\to \mathcal{M}$ such that its restriction $\Ret_x$ to $T_{x}\mathcal{M}$ satisfies $\Ret_{x}(0) = x$ and $\d\Ret_{x}(0) = \id_{T_{x}\mathcal{M}}$, see~\cite[Def. 4.1.1]{absil2009optimization}. In our case, we need to slightly adapt the definition to deal with the cone structure.
\begin{definition}\label{def:compatibleretraction}
We say that $\Ret: \Omega \times (\RR\times T\Theta)\to \Omega$ is a retraction compatible with the cone structure, if it satisfies the following:
\begin{enumerate}[(i)]
\item \emph{(Retraction property)} It is a proper retraction on $\Omega^* \coloneqq \RR_+^* \times \Theta$. It is not necessarily defined everywhere but there exists $C>0$ such that $\Ret_{(r,\theta)}(\delta r, \delta \theta)$ is defined as long as $\max\{ \vert \delta r\vert/r,\Vert \delta \theta\Vert_\theta\}  < C$.
\item \emph{(Zero preserving)} It satisfies $\Ret_{(0,\theta)}(\delta r,\delta \theta) = (0,f( \theta, \delta \theta))$ for some arbitrary measurable $f$.
\item \emph{(Homogeneity)} For any $r,\tilde r \in \RR_+^*$, $\theta \in \Theta$, $\delta r \in \RR$ and $\delta \theta \in T_\theta\Theta$ satisfying $\max\{ \vert \delta r\vert/r,\Vert \delta \theta\Vert_\theta\}  < C$, denoting $(r_1,\theta_1)= \Ret_{(r,\theta)}(r\delta r, \delta \theta) $ and $(r_2,\theta_2)= \Ret_{(\tilde r,\theta)}(\tilde r \delta r, \delta \theta)$, then $\theta_1=\theta_2$ and $\tilde r \cdot r_1 =r \cdot r_2$.
\end{enumerate}
\end{definition}

 These properties are satisfied in the following  examples, where $\widetilde \Ret$ denotes any retraction defined on $\Theta$ (we give them names for future reference):
\begin{itemize}
\item[--] the \emph{canonical} retraction $\Ret_{(r,\theta)}(\delta r, \delta \theta) = (r+\delta r, \widetilde \Ret_{\theta}(\delta \theta))$  (here $C=1$);
\item[--] the \emph{mirror} retraction $\Ret_{(r,\theta)}(\delta r, \delta \theta) = (r\exp(\delta r/r), \widetilde \Ret_{\theta}(\delta \theta))$, which allows to recover a version of mirror descent when $\delta \theta=0$ (here $C=+\infty$);
\item[--] the \emph{induced} retraction when $\Theta$ is the $d$-sphere, which is the retraction induced by the isometric embedding into $\RR^{d+1}$, see Section~\ref{sec:conic}. It is defined as $\Ret_{(r,\theta)} (\delta r,\delta \theta)= (\Vert u\Vert, u/\Vert u\Vert )$ where $u = r\theta+\theta \delta r+r\delta \theta \in \RR^{d+1}$ (here $C=1$). With this retraction, the iterates of gradient descent on $\Omega$ with the cone metric can be identified with the iterates of (Euclidean) gradient descent in $\RR^{d+1}$.
\end{itemize}

\paragraph{Gradient descent in $\Pp_2(\Omega)$.} Given a retraction $\Ret$ compatible with the cone structure, we define the gradient descent as follows. Let $\mu_0\in \Pp_2(\Omega)$ and for $k\in \NN$ define recursively
\begin{align}\label{eq:GD}
\mu_{k+1} = (T_k)_\# \mu_k
\end{align}
where $T_k(r,\theta) = \Ret_{(r,\theta)}(-2\alpha rJ'_{\nu_k}(\theta),-\beta \nabla J'_{\nu_k}(\theta))$ and $\nu_k =\h \mu_k$. The notation $\#$ stands for the pushforward operator\footnote{The pushfoward measure $T_\# \mu$ is characterized by $\int \psi\d(T_\#\mu) = \int (\psi\circ T) \d\mu$ for any continuous function $\psi$.}. When $\mu_0$ is a finite discrete probability measure with uniform weights, this gives Algorithm~\ref{alg:gradientbased}, which is a gradient descent for $F_m$ in the cone metric.
\begin{algorithm}\caption{Conic Particle Gradient Descent}\label{alg:gradientbased}
\begin{enumerate}
\item let $\alpha$ and $\beta$ be positive step-sizes and $\Ret$ a retraction on $\Omega$ compatible with the cone structure (Definition~\ref{def:compatibleretraction}).
\item define an initial distribution of $m$ particles weights-locations $(r_i^{(0)}, \theta_i^{(0)})_{i=1}^m$.
\item define for $k=0,1,\dots$ until a stopping criterion is satisfied 
\begin{align*}
(r_i^{(k+1)},\theta_i^{(k+1)}) \leftarrow \mathrm{Ret}_{(r_i^{(k)},\theta_i^{(k)})}(- 2\alpha  r_i^{(k)} J'_{\nu^{(k)}}(\theta^{(k)}_i), - \beta \nabla J'_{\nu^{(k)}}(\theta^{(k)}_i) ) \text{ for $i\in \{1,\dots,m\}$}
\end{align*}
where $J'_\nu$ is given by Eq.~\eqref{eq:firstvariation} and $\nu^{(k)} \coloneqq \frac1m \sum_{j=1}^m (r_j^{(k)})^2\delta_{\theta_j^{(k)}}$.
\end{enumerate}
\end{algorithm}

\paragraph{Transport-growth interpretation.} Just like the continuous-time gradient flow, the discrete time gradient descent has a corresponding projected dynamics in $\Mm_+(\Theta)$. Here the equivalence also relies on the properties of compatible retractions.
\begin{proposition}\label{prop:projectedGD}
Under {\sf (A1-2)}, let $\Ret$ be a retraction compatible with the cone structure and let $\mu_{k+1} = (T_k)_\#\mu_k$  for some $\mu_k \in \Pp_2(\Omega)$. Let $(T_k^r(\theta),T_k^\theta(\theta))\coloneqq T_k(1,\theta)$. Then, the projected iterates $(\nu_{k+1},\nu_k) \coloneqq  (\h\mu_{k+1}, \h\mu_k)$ satisfy
\begin{equation}\label{eq:projectedGD}
\nu_{k+1} = (T_k^\theta)_\# ((T_k^r)^2\nu_k).
\end{equation}
\end{proposition}
\begin{proof}
First, remark that by Property~(i) of Definition~\ref{def:compatibleretraction}, $T_k$ is well-defined if $\max\{\alpha,\beta\}$ is small enough and that $T_k\in L^2(\mu_k;\Omega)$ so $\mu_{k+1}\in \Pp_2(\Omega)$. For any continuous function $\psi:\Theta\to \RR$, using Properties (ii)-(iii) of Definition~\ref{def:compatibleretraction}, we get
\begin{align*}
\int \psi \d\nu_{k+1} &=\int r^2 \psi(\theta) \d((T_k)_\# \mu_k)(r,\theta) = \int (r T^r_k(\theta))^2\psi(T^\theta_k(\theta))\d\mu_k(r,\theta) = \int (T_k^r)^2 ( \psi\circ T^\theta_k) \d\nu_k 
\end{align*}
which proves the claim.
\end{proof}

\paragraph{Descent property of conic particle gradient descent.}
The following lemma shows that, for sufficiently small step-sizes, the iterates~\eqref{eq:GD} are well-defined and monotonously decrease the objective. As usual in optimization, this property is useful to convert results on gradient flows into results on gradient descent.
\begin{lemma}[Descent property]\label{lem:sufficientdescrease}
Assume {\sf(A1-2)} and let $\Ret$ be a retraction compatible with the cone structure  (Definition~\ref{def:compatibleretraction}). For any $J_{\max}\geq J^\star$, there exists $\eta_{\max}>0$ such that if $\nu_0 \in \Mm_+(\Theta)$ satisfies $J(\nu_0)\leq J_{\max}$ then the gradient descent iteration with $\max\{\alpha, \beta\} \leq \eta_{\max}$ is well defined for all $k\geq 0$ and satisfies
\begin{align*}
J(\nu_{k+1})- J(\nu_k) \leq -\frac{1}2 \Vert g_{\nu_k}\Vert^2_{L^2(\nu_k)} && \text{where} && \Vert g_{\nu}\Vert^2_{L^2(\nu)} \coloneqq \int \left(4\alpha \vert J'_{\nu}(\theta)\vert^2 + \beta\Vert \nabla J'_{\nu}(\theta) \Vert^2_{\theta}\right)\d\nu(\theta).
\end{align*}
\end{lemma}
\begin{proof}
Let us first look at one step starting from $\nu_k\in \Mm_+(\Theta)$. By Property~(i) of Definition~\ref{def:compatibleretraction},  there exists $\eta_{\max}>0$ such that this iteration is well-defined as long as $\max\{\alpha,\beta\} \leq \eta_{\max}$. We first consider $\nu_k(\Theta)$ and $\Vert J'_{\nu_k}\Vert_{\mathcal{C}^2}$ as constants, where $\Vert \phi\Vert_{\mathcal{C}^2} = \max\{ \Vert \phi\Vert_{\infty}, \Vert \nabla \phi\Vert_{\infty}, \Vert \nabla^2 \phi\Vert_{\infty}\}$ (we will see later that they can be upper bounded independently of the iteration $k$). With the notations of Proposition~\ref{prop:projectedGD},  we  have $\nu_{k+1} = (T^\theta_k)_\#((T^r_k)^2\nu_k)$ where $T_k^r(\theta)=1-2\alpha J'_{\nu_k}(\theta) +O(\alpha^2J'_{\nu_k}(\theta)^2)$ and, in normal coordinates, $T_k^\theta(\theta) = \theta - \beta \nabla J'_{\nu_k}(\theta) +O(\beta^2 \Vert \nabla J'_{\nu_k}(\theta)\Vert^2)$ where the hidden constants are uniform in $\theta$. It follows that for any twice continuously differentiable $\psi \in \mathcal{C}^2(\Theta;\RR)$, it holds
\begin{align*}
\int \psi\d(\nu_{k+1}-\nu_k) 
&= \int ((T^r_k(\theta))^2-1)\psi(T^\theta_k(\theta))\d\nu_k(\theta) + \int   (\psi( T^\theta_k(\theta))-\psi(\theta))\d\nu_k(\theta)\\
 & = -\int \left(4\alpha \psi \cdot J'_{\nu_k} + \beta \nabla \psi \cdot \nabla J'_{\nu_k} \right)\d\nu_k+\Vert \psi\Vert_{\mathcal{C}^2}  \Vert g_{\nu_k}\Vert^2_{L^2(\mu_k)} O\left( \max\{\alpha,\beta\} \right).
\end{align*}
In particular, using this expression with $\psi_f(\theta) = \langle \phi(\theta),f\rangle$  where $\Vert f \Vert\leq 1$ (which have uniformly bounded norms $\Vert \psi_f\Vert_{\mathcal{C}^2}$ under our assumptions), we get that 
$$
\Big\Vert \int \phi \d(\nu_{k+1}-\nu_k)\Big\Vert^2  = \sup_{\Vert f\Vert\leq 1} \Big\Vert \int \psi_f \d(\nu_{k+1}-\nu_k)\Big\Vert^2 =  O(\max\{\alpha,\beta\} \Vert g_{\nu_k}\Vert^2_{L^2(\mu_k)}).
$$
By a first order expansion of $R$, we have for $f,f'\in \Ff$, $R(f')-R(f) = \langle f'-f, \nabla R(f)\rangle + O(\Vert f'-f\Vert^2)$. Thus, using the expression of $J'_{\nu}$ from Eq.~\eqref{eq:firstvariation}, it follows
\begin{align*}
J(\nu_{k+1})-J(\nu_k) & =  \Big\langle \int \phi \d(\nu_{k+1} - \d\nu_k), \nabla R\Big(\int \phi \d\nu_k\Big)\Big\rangle\\
& \qquad + \lambda \int (\d\nu_{k+1}-\d\nu_k)+  O(\max\{\alpha,\beta\} \Vert g_{\nu_k}\Vert^2_{L^2(\mu_k)})\\
&=  \int J'_{\nu_k}\d (\nu_{k+1}-\nu_k) + O(\max\{\alpha,\beta\} \Vert g_{\nu_k}\Vert^2_{L^2(\mu_k)}) \\
& = (-1 + O(\max\{\alpha,\beta\})) \Vert g_{\nu_k}\Vert^2_{L^2(\mu_k)}.
\end{align*}
So there exists $\eta_{\max}$ such that if $\max\{\alpha,\beta\}\leq  \eta_{\max}$, we have $J(\nu_{k+1})-J(\nu_k)\leq -\frac12 \Vert g_{\nu_k}\Vert^2_{L^2(\mu_k)}$. Finally, since we have assumed that $\lambda>0$ and $\nabla R$ is bounded on sublevel sets, the quantities $\sup_{J(\nu)\leq J(\nu_k)} \nu(\Theta)$ and $\sup_{J(\nu)\leq J(\nu_k)} \Vert J'_\nu\Vert_{\mathcal{C}^2}$ are finite. By the decrease property we just proved, these quantities decrease after one iteration if $\max\{\alpha,\beta\} \leq \eta_{\max}$. So $\eta_{\max}$, which depends on these quantities, can be chosen independently of $k\geq 0$. 
\end{proof}

\section{Exponential local convergence}\label{sec:local}

We now proceed to the theoretical analysis of the projected gradient flow~\eqref{eq:projectedGF} and projected gradient descent~\eqref{eq:projectedGD} in the conic setting. In light of Propositions~\ref{prop:projectedGF} and~\ref{prop:projectedGD}, these dynamics correspond to the gradient flow and gradient descent of $F$, seen through the projection operator~$\h$.

\subsection{Non-degeneracy assumptions}\label{sec:theoryprelim}
In order to derive global optimality conditions, we assume the following.
\begin{itemize}
\item [\sf(A3)] The loss $R$ is convex.
\end{itemize}
Commonly used losses that satisfy the smoothness and convexity conditions are the square loss and the logistic loss. Under this assumption, we have existence of minimizers and a global optimality condition.
\begin{proposition}[Optimality condition]\label{prop:optimalitycondition} Under~{\sf(A1)} and~{\sf(A3)}, problem~\eqref{eq:objective} admits minimizers. Moreover, a measure $\nu^\star \in \Mm_+(\Theta)$ is a minimizer if and only if it holds $J'_{\nu^\star}(\theta)\geq 0$ for all $\theta\in \Theta$ and $J'_{\nu^\star}(\theta)=0$ whenever $\theta$ in the support of $\nu^\star$.
\end{proposition}
\begin{proof}
As $\lambda$ is assumed positive, the sublevel sets of $J$ on $\Mm_+(\Theta)$ are bounded in total variation, and are thus weakly pre-compact. It follows that any minimizing sequence for $J$ admits at least one weak limit point $\nu^\star$, which is a minimizer of~\eqref{eq:objective} since $J$ is weakly continuous. The stated optimality condition is equivalent to having $\int J'_{\nu^\star}\d (\nu -\nu^\star) \geq 0$ for all $\nu \in \Mm_+(\Theta)$. The latter is a sufficient optimality condition since by convexity of $J$, $J(\nu)-J(\nu^\star) \geq \int J'_{\nu^\star}\d (\nu -\nu^\star)$. It is also necessary since  it holds $\frac{d}{d\epsilon} J((1-\epsilon)\nu^\star +\epsilon \nu)\vert_{\epsilon = 0^+} = \int J'_{\nu^\star} \d(\nu-\nu^\star)$.
\end{proof}

\paragraph{Sparse minimizer.}
Our local analysis requires sparsity of the minimizers of the objective $J$,  which can be guaranteed a priori  in several settings (e.g.~\cite{duval2015exact,boyer2019representer}).

\begin{itemize}
\item [\sf(A4)] Problem~\eqref{eq:objective} admits a unique global minimizer on $\Mm_+(\Theta)$ which is of the form $\nu^\star = \sum_{i=1}^{m^\star} r_i^2\delta_{\theta_i}$ with $\nu^\star(\Theta)>0$. We denote $f^\star \coloneqq \int \phi \d\nu^\star = \sum_{i=1}^{m^\star}r_i^2\phi(\theta_i)$.
\end{itemize}
Without loss of generality, we assume $r_i>0$ for all $i$ and $\theta_i \neq \theta_{i'}$ whenever $i\neq i'$, so that $(r_i,\theta_i)_{i=1}^{m^\star}$ is uniquely well-defined, up to re-ordering. Let us fix from now on normal coordinates frames on the neighborhood of each $\theta_i$. This allows to identify tensors at $\theta_i$ with their expression in coordinates and also induces a set of coordinates on the direct sum of the tangent spaces $T_{\theta_i}\Theta$, which is of dimension $m^\star\times d$. 
\paragraph{Kernels and non-degeneracy.}
We define the \emph{global kernel} $K\in \RR^{(m^\star\times (1+d))^2}$ by
\[
K_{(i,j),(i',j')} \coloneqq \langle r_i \bar \nabla_j \phi(\theta_i), r_{i'} \bar \nabla_{j'}\phi(\theta_{i'})\rangle_{d^2\! R_{f^\star}}
\]
where $\bar\nabla \phi \coloneqq (2 \alpha \phi, \beta \nabla \phi)$ can be interpreted as the gradient of $\h\phi$ at $(1,\theta)$. Remark that $K$ is defined via the quadratic form associated to the Hessian of $R$ at $f^\star$. This interaction kernel $K$ appears naturally in the various statistical and optimization analysis of the minimization problem under consideration~\cite{duval2015exact,traonmilin2018basins}. We also use the notation for the \emph{local kernels}  for $i\in 1,\dots,m^\star$
\[
H_i\coloneqq \beta^2 \nabla^2 J'_{\nu^\star}(\theta_i)
\]
expressed in local coordinates.  In order to simplify notations, we concatenate these matrices in a large matrix $H\in \RR^{(m^\star\times (1+d))^2}$ of the same size as $K$ defined as
 \[
 H_{(i,j),(i,j')} =\begin{cases}
 \beta^2  \nabla^2_{j,j'} J'_{\nu^\star}(\theta_i) &\text{if $i=i'$ and $j,j'\geq 1$},\\
  0 &\text{ if $j=0$ or $j'=0$}.
  \end{cases}
 \]
where here and in the proofs, we use $0$ to label the $r$'s coordinate. The local analysis will be carried under the following non-degeneracy assumptions.
\begin{itemize}
\item [\sf(A5)] The minimizer $\nu^\star$ is non-degenerate in the sense that $\nabla^2 R(f^\star)$ is positive definite and, calling $\sigma_{\min}(A)$ the smallest singular value of a linear operator $A$, we have global curvature  $\sigma_{\min}(K)>0$, local curvature  $\sigma_{\min}(H) = \min_{i} \sigma_{\min}(H_i)>0$, and strict slackness, i.e.\ the only points where $J'_{\nu^\star}$ vanishes are $(\theta_i)_{i=1}^{m^\star}$.
\end{itemize}

The first property is always satisfied if $R$ is strictly convex. The second property is satisfied when the kernel associated to the feature function $\bar \nabla \phi$ is positive definite. The last two assumptions unfortunately depend on an \emph{a priori} unknown object $J'_{\nu^\star}$, but are often required to perform analysis of Problem~\eqref{eq:objective}~\cite{flinth2019linear,duval2015exact}. Yet, in some cases, they can be guaranteed to hold, see e.g.~\cite{poon2018geometry,tang2013compressed}. In spite of this drawback, the local analysis leads to interesting qualitative insights on the dynamics in practice, see Section~\ref{sec:numerics}.

\subsection{Convergence in $\Mm_+(\Theta)$}
A first consequence of these assumptions is that convergence in value implies convergence to minimizers. The distance on $\Mm_+(\Theta)$ that naturally appears in the analysis is the Wasserstein-Fisher-Rao, a.k.a.\ Hellinger-Kantorovich metric $\widehat W_2$, which is the extension of the Wasserstein $W_2$ metric to unnormalized measures. It admits many equivalent definitions~\cite{liero2018optimal, kondratyev2016new, chizat2018unbalanced}, the most suitable to our context being~\cite[
Thm. 7.20]{liero2018optimal}
\[
\widehat W_2(\nu_1,\nu_2) \coloneqq \min \left \{W_2(\mu_1,\mu_2)\;;\; (\mu_1,\mu_2)\in \Pp_2(\Omega)^2 \text{ satisfy } (\h\mu_1,\h\mu_2)=(\nu_1,\nu_2)\right\}
\]
where the Wasserstein distance on $\Omega$ is defined relative to the cone metric (in this paragraph, with $\alpha=\beta=1$). The proof of the following result involves the construction of a transport map in the lifted space $\Pp_2(\Omega)$ and is postponed to Appendix~\ref{app:distanceinequality}.
\begin{proposition}\label{prop:distanceinequality}
Under {\sf (A1-5)}, for all $J_{\max}\geq J^\star$, there exists $C,C'>0$, such that if $\nu\in \Mm_+(\Theta)$ satisfies $J(\nu)\leq J_{\max}$ then $\Vert \nu - \nu^\star\Vert^*_{\BL} \leq C \widehat W_2(\nu, \nu^\star) \leq C'(J(\nu)-J^\star)^{\frac12}$.
\end{proposition}

\subsection{Sharpness of the objective}
 Our first main result is a lower bound on the squared norm of the gradient in terms of the sub-optimality gap, an inequality known as sharpness, or Polyak-\L{}ojasiewicz inequality~\cite{polyak1963gradient,karimi2016linear}, which is a special case of \L{}ojasiewicz gradient inequality. It involves the $L^2(\nu)$ norm of the gradient, which we denote for $\nu=\h\mu$ by
 \[
 \Vert g_{\nu}\Vert^2_{L^2(\nu)} \coloneqq \int_{\Omega} \left( \frac1\alpha \vert 2\alpha rJ'_{\nu}(\theta)\vert^2 + \frac {r^2}{\beta} \Vert \beta \nabla J'_{\nu}\Vert_\theta^2\right)\d\mu(r,\theta) = \int_\Theta \left( 4 \alpha \vert J'_\nu(\theta)\vert^2 +\beta \Vert \nabla J'_\nu(\theta)\Vert^2_\theta \right)\d\nu(\theta).
 \]

\begin{theorem}[Sharpness]\label{th:gradientinequality} Under {\sf(A1-5)}, there exists $J_0>J^\star$ and $\kappa_0>0$, such that for all $\nu\in \Mm_+(\Theta)$ satisfying $J(\nu) \leq  J_0$ and $\alpha,\beta>0$, one has
\[
\frac12 \Vert g_\nu\Vert^2_{L^2(\nu)}\geq \kappa_0\, \min\{\alpha,\beta\}\, (J(\nu)-J^\star).
\]
\end{theorem}
While the objective is non-convex in the Wasserstein geometry and has typically an infinity of bad stationary points, this inequality guarantees exponential convergence to global minimizers of various gradient-based dynamics as long as their initialization $\nu_0$ has a small enough objective value. Crucially, the specific structure of $\nu$ does not matter, beyond the fact that is is close enough to optimality: it applies indifferently to discrete and absolutely continuous measures.
Once Theorem~\ref{th:gradientinequality} is established, it is straightforward to prove exponential convergence of gradient flow and gradient descent.

\begin{corollary}[Local convergence of gradient flow]\label{cor:localGF}
Under {\sf (A1-5)}, let $J_0$ and $\kappa_0$ be given by Theorem~\ref{th:gradientinequality}. Consider $(\nu_t)_{t\geq 0}$ a projected gradient flow for $J$ as in Eq.~\eqref{eq:projectedGF}. If $J(\nu_0)\leq J_0$ then 
\begin{align*}
J(\nu_t) - J^* &\leq  (J(\nu_0) - J^*) \exp(-2\kappa_0 \min\{\alpha,\beta\}  t).
\end{align*}
\end{corollary}
\begin{proof}
By Theorem~\ref{th:gradientinequality} and direct computations, one has 
\[
\frac{d}{dt} (J(\nu_t)-J^\star) = \int_{\Omega} J'_{\nu_t} \d(\partial_t \nu_t) = - \Vert g_{\nu_t}\Vert^2_{L^2(\nu_t)} \leq -2\kappa_0\min\{\alpha,\beta \} (J(\nu_t)-J^\star)
\]
and the result follows by Gr\"onwall's lemma.
\end{proof}

\begin{corollary}[Local convergence of gradient descent]\label{cor:localGD}
 Assume {\sf(A1-5)}, let $J_0$ and $\kappa_0$ be given by Theorem~\ref{th:gradientinequality}, and let $\Ret$ be a retraction compatible with the cone structure (Definition~\ref{def:compatibleretraction}). There exists $\eta_{\max}>0$ such that for  any projected gradient descent $(\nu_k)_{k\geq 0}$ for $J$ following recursion~\eqref{eq:GD}, if $J(\nu_0)\leq J_0$ and $\max\{\alpha ,\beta \}\leq \eta_{\max}$, then
 \begin{align*}
 J(\mu_k)-J^\star \leq (J(\nu_0) - J^\star) \left(1-\kappa_0 \min\{\alpha,\beta\} \right)^k.
 \end{align*}
\end{corollary}
\begin{proof}
By Lemma~\ref{lem:sufficientdescrease}, there exists $\eta_{\max}$ such that if $\max\{\alpha,\beta\} \leq \eta_{\max}$, then $J(\nu_{k+1}) - J(\nu_k)  \leq - \frac12 \Vert g_{\nu_k}\Vert^2_{L^2(\nu_k)}$. Combining this inequality with Theorem~\ref{th:gradientinequality}, one has $J(\nu_{k+1})- J(\nu_k)\leq -\kappa_0 \min\{\alpha,\beta\} (J(\nu_k)-J^\star)$. Rearranging the terms, we get $J(\nu_{k+1})-J^\star \leq (1-\kappa_0\min\{\alpha,\beta\})(J(\nu_k)-J^\star)$ and the result follows by recursion.
\end{proof}

\subsection{Proof strategy for the sharpness theorem}
The proof of Theorem~\ref{th:gradientinequality}, in Appendix~\ref{app:sharpness}, is based on a local expansion of $J(\nu)$ in terms of some local moments of $\nu$. For a radius $\tau>0$ (that shall be fixed at some small enough value in the course of the proof), we define the sets  for $i\in \{1,\dots,m^\star\}$,
\begin{align*}
\Theta_i \coloneqq  \{ \theta \in \Theta \;;\; \dist(\theta,\theta_i) < \tau \}.
\end{align*}
We assume that $\tau$ is smaller than $1$ and small enough so that these sets together with $\Theta_0 \coloneqq \Theta \setminus \cup_{i=1}^{m^\star} \Theta_i$ form a partition of $\Theta$ and that the exponential map at $\theta_i$ has injectivity radius larger than $\tau$, for $i\in \{1,\dots,m^\star\}$. We then say that $\tau$ is an \emph{admissible} radius. 
\begin{definition}[Local moments]\label{def:localfeatures}
Given an admissible radius $\tau>0$ and a measure $\nu\in \Mm_+(\Theta)$, we define for $i\in \{0,\dots,m^\star\}$ the local masses $\bar r_i^2 = \nu(\Theta_i)$ and the local means $\bar \theta_i \coloneqq \frac1{\bar r_i^2} \int_{\Theta_i} \theta\d\nu(\theta)$ if $\nu(\Theta_i)>0$ and $\bar \theta_i = \theta_i$ otherwise. Finally, we define for $i\in \{1,\dots,m^\star\}$ the \emph{weighted biases}
\begin{align*}
b_i^r \coloneqq \frac{\bar r_i^2 -r_i^2}{2\alpha r_i},&&
b_i^\theta \coloneqq \frac{\bar r_i^2}{\beta r_i}(\bar \theta_i-\theta_i), &&
b_i = (b_i^r,b_i^\theta),
\end{align*}
and the \emph{weighted covariances} $\Sigma_i \coloneqq \frac1{\bar r_i^2 \beta^2} \int_{\Theta_i} (\theta - \bar \theta_i)(\theta - \bar \theta_i)^\intercal \d\nu(\theta)$.
\end{definition}

If $\nu$ has only $1$ atom in each $\Theta_i$ then its spatial coordinate is $\bar \theta_i$ and $\Sigma_i=0$. When moreover $\nu(\Theta_0)=0$, the optimization reduces to a more classical gradient flow in $\Omega^{m^\star}$ which local behavior has already been studied~\cite{traonmilin2018basins,flinth2019linear}, but obtaining measures of this form is typically almost as hard as solving the original problem. This decomposition can be reminiscent of proof techniques used to study log-Sobolev inequalities (another type of sharpness inequality in Wasserstein space~\cite{blanchet2018family}) in the small temperature regime~\cite{menz2014poincare}.

It turns out that the local moments of Definition~\ref{def:localfeatures} are sufficient to characterize the behavior of $J$ near optimality. In particular, we have the following approximations for $J$ and its gradient around optimality. These formulas are obtained as an intermediate step in the proof of Theorem~\ref{th:gradientinequality} and follow by combining the bounds of Proposition~\ref{prop:localexpansion} and Proposition~\ref{prop:gradientestimate} with Lemma~\ref{lem:distancebound}.

\begin{proposition}[Local expansion]\label{prop:expansions}
Assuming {\sf (A1-5)}, for any $\nu\in \Mm_+(\Theta)$ it holds
 \begin{align}\label{eq:expansion}
J(\nu)-J^\star &= \frac12 b^\intercal (K+ H) b + \frac12 \sum_{i=1}^{m^\star} r_i^2 \tr(\Sigma_i H_i)  + \int_{\Theta_0} J'_{\nu^\star} \d\nu + \mathrm{err}(\tau,J(\nu)-J^\star)\\
\intertext{and}
\frac12 \Vert g_\nu\Vert^2_{L^2(\nu)} &= \frac12 b^\intercal (K+ H)^2 b + \frac12 \sum_{i=1}^{m^\star} r_i^2 \tr(\Sigma_i H_i^2)  + \frac12 \Vert g_\nu\Vert^2_{L^2(\nu\vert_{\Theta_0})} + \mathrm{err}(\tau,J(\nu)-J^\star)
 \end{align}
 where $\Vert g_\nu\Vert^2_{L^2(\nu\vert_{\Theta_0})}  = \int_{\Theta_0} (\alpha \vert J'_\nu\vert^2+\beta \Vert\nabla J'_\nu\Vert^2_\theta )\d\nu$ and $\mathrm{err}(\tau,\Delta) = O(\Delta \tau+ \Delta^{\frac32}\tau^{-6})$.
 \end{proposition}

\subsection{Discussion on the local behavior}
 
Let us now explore what the expansion from Proposition~\ref{prop:expansions} teaches us about the local behavior of the dynamics. In order to simplify the discussion, let us fix a small admissible radius $\tau_0$ and ignore the error terms in Proposition~\ref{prop:expansions}.

\paragraph{Effect of over-parameterization.} When there is no over-parameterization ($m=m^\star$) and we have a single particle in the neighborhood $\Theta_i$ of each optimal particle, then there is no local variance: $\Sigma_i=0$ for $i=1,\dots,m^\star$. In this case, we recover the Taylor expansion of $F_{m^\star}$ around its minimizer
\[
J(\nu) -J^\star \approx \frac12 b^\intercal (K+ H) b,
\]
and the local convergence rate is dictated by the conditioning of $(K+H)$. Now, for an arbitrary over-parameterization i.e.\ $\nu \in \Mm_+(\Theta)$ but with the support of the solution approximately identified, i.e.\ $\nu(\Theta_0)=0$, the objective is still entirely characterized locally by the local moments of $\nu$, since
\[
J(\nu) -J^\star \approx \frac12 b^\intercal K b + \frac12 \sum_{i=1}^{m^\star} b_i^\intercal H_i b_i + \frac12 \sum_{i=1}^{m^\star} r_i^2 \tr(\Sigma_i H_i).
\]
This expression gives a clear picture of the energy landscape, so let us comment on it. If we think of the particles in $\Theta_i$ as a cluster, then the first term consists in a global interaction between the clusters, which only depends on the biases of each cluster relatively to their respective ground truth particles. The two other terms are local interactions within each cluster, which are due to the local curvature of $J'_{\nu^\star}$ at each $\theta_i$. Note in particular that the only term in this expansion that penalizes the variance of each cluster $\Sigma_i$ consists of local interactions.

\paragraph{Effect of the regularization parameter.} In this paper, the assumption that $\lambda$ is non-zero is not crucial as such. Instead the crucial assumption for the local analysis is~{\sf(A5)}. Still, this assumption is intimately connected to the regularization: in the signed case (detailed in Appendix~\ref{app:signed}), it is necessary that $\lambda>0$ to have~{\sf(A5)}, because with $\lambda=0$, the minimizer $\nu^\star$ is a global minimizer in the space of signed measures and thus the global optimality condition $J'_{\nu^\star}=0$ holds.
In fact, a finer analysis of the behavior as $\lambda\to 0$ is possible in the signed case: it can be shown that one has (emphasizing the dependency in $\lambda$ in the notation): 
\begin{align*}
K_\lambda = K_0 +o(\lambda),&& 
H_{\lambda} = \lambda H_0 + o(\lambda), &&
J'_{\nu*,\lambda} = \lambda J'_0 +o(\lambda),
\end{align*}
for some $K_0$, $H_0$ and $J'_{0}$~\cite[Prop.~1 and Thm.~2]{duval2015exact} (where the result is proved for $R$ being the square loss but can be directly generalized to $R$ smooth and strongly convex around the minimizer).
Under the assumption that $J'_0$ is non-degenerate in the sense of {\sf (A5)}, as soon as $\nu(\Theta_0)>0$ or $\Sigma_i\neq 0$ for some $i\in \{1,\dots,m^\star\}$, the local rate $\kappa_0$ is thus of order $\lambda$ and for $\lambda=0$, the exponential convergence rate is lost. This shows that regularization is necessary for fast local convergence in the signed case, and in particular -- remembering the previous paragraph -- for the variance of each cluster of particles to vanish quickly. Note that it is an open question to even show local convergence when {\sf(A5)} does not hold.

\paragraph{Choice of the metric and conditioning.} While our statements, in particular Corollary~\ref{cor:localGD}, seem to imply that it is best to choose $\alpha=\beta$, this is in fact just an artefact of the way the upper bounds are presented, with some hidden constants. Instead, these parameters should be chosen, as usual, so as to make the local expression of $J$ above well-conditioned. Without additional information, a possible heuristic is to make the block diagonal matrix $\diag(K+H,H)$ well-conditioned by choosing $(\alpha,\beta)$ satisfying $2\alpha \Vert \phi\Vert_\infty \approx \beta \Lip(\phi)$. 

\paragraph{Polynomial dependency.}
It can be seen from the proof of Theorem~\ref{th:gradientinequality} that $(J_0-J^\star)^{-1}$ and $\kappa_0^{-1}$ depend  polynomially on the characteristics of the problem, which are the regularization $\lambda$, the regularity parameters of $\phi$ and $R$, the ratio $\max_i r_i/\min_i r_i$, the inverses of the $\sigma_{\min}(\nabla^2R(f^\star))$, $\sigma_{\min}(H)$, $\sigma_{\min}(K)$ and finally the quantity $v^\star$ that quantifies the strict slackness assumption, in the following sense: $v^*>0$ is such that for any local minimum $\theta$ of $J'_{\nu^\star}$, either $\theta = \theta_i$ for some $i\in \{1,\dots,m^\star\}$ or $J'_{\nu^\star}(\theta)\geq v^*$.

\section{Quantitative global convergence}\label{sec:global}
There are several convex optimization-based algorithms that are known to return approximate minimizers of $J$ which are mixtures of atoms (with typically $m>m^\star$) with a guaranteed complexity, see Section~\ref{sec:related}. Starting from any such approximate minimizer, the results of the previous section imply that conic particle gradient descent converges exponentially fast to minimizers of $J$. However, such a ``two-algorithms'' approach comes with a drawback: one has to decide when to switch from one algorithm to another. In this section, we show that it is possible to reach global optimality by only performing non-convex gradient descent. This is true under two main conditions: (i) the initialization samples $\Theta$ densely enough, and (ii) the ratio $\beta/\alpha$ is small, at least in the early stages of the algorithm.

\subsection{Statement of the main results}

In order to state the condition on the initialization, we first choose a reference measure $\rho \in \Mm_+(\Theta)$ with a smooth positive density, also denoted by $\rho$, which represents our prior knowledge about the solution $\nu^\star$. We introduce the quantity (analogous to a log-likelihood)
\[
\bar\Hh(\nu^\star,\rho) \coloneqq \sum_{i=1}^{m^\star} r_i^2 \log\left(\frac{r_i^2}{\rho(\theta_i)}\right) - \nu^\star(\Theta)+\rho(\Theta).
\]
It quantifies how good is $\rho$ as a prior for the unknown minimizer $\nu^\star$ and we will see that our convergence bounds are better when $\bar \Hh(\nu^\star,\rho)$ is smaller. If nothing is known about the optimal positions $\theta_i$, we should choose $\rho$ as a uniform density $\alpha \!\vol$ over $\Theta$ for some $\alpha>0$. Minimizing $\bar \Hh(\nu^\star,\alpha \vol)$ in $\alpha$ suggests to choose $\alpha = \frac{\nu^\star(\Theta)}{\vol(\Theta)}$.

To obtain an implementable algorithm, we then discretize $\rho$ and consider an initialization $\nu_0\in \Mm_+(\Theta)$ which is close to $\rho$ in the $W_\infty$ distance (our statements do not require $\nu_0$ to be discrete but this is necessary to obtain an implementable algorithm). We now state our main theorem.

\begin{theorem}[Global convergence of gradient flow]\label{thm:convergenceflow}
Under {\sf (A1-5)}, let $J_0$ and $\kappa_0$ given by Theorem~\ref{th:gradientinequality}, let $\rho = \rho\!\vol \in \Mm_+(\Theta)$ an absolutely continuous reference measure with $\log \rho$ $L$-Lipschitz and let  $B_{\nu_0} = \sup_{J(\nu)\leq J(\nu_0)} \Vert J'_{\nu}\Vert_{\BL}$, where $\nu_0\in \Mm_+(\Theta)$ is the initialization. For any $0<\epsilon\leq 1/2$, there exists $C_\epsilon >0$ that only depend on $\epsilon$ and bounds on the curvature of $\Theta$ such that if it holds $\beta/\alpha \leq (4B_{\nu_0}/\max\{1,L\})^2$,
\begin{align*}
W_\infty(\nu_0,\rho)\leq \frac{J_0-J^\star}{2B_{\nu_0}\nu^\star(\Theta)} &&\text{and}&&
\frac{\beta}{\alpha} \leq\left(  \frac{J_0-J^\star}{2(4B_{\nu_0})^{\epsilon}\left(\bar \Hh(\nu^\star,\rho) +B^2_{\nu_0} +C_\epsilon \nu^\star(\Theta)d \right)}\right)^\frac{2}{1-\epsilon}
\end{align*}
 then the projected gradient flow $(\nu_t)_{t\geq 0}$ initialized with $\nu_0$ converges to the global minimizer $\nu^\star$. Denoting $t_0 = 1/\sqrt{\alpha\beta}$ it satisfies, for $t\geq t_0$,
\begin{align*}
 J(\nu_t)-J^\star  \leq  (J_0 - J^*) \exp(-2\kappa_0\min\{\alpha,\beta\} (t-t_0)).
 \end{align*}
\end{theorem}
We also state a similar result for gradient descent, but without tracking the constants. The proof follows the same lines as that of Theorem~\ref{thm:convergenceflow} and is given in Appendix~\ref{app:gradientdescent}.
\begin{theorem}[Global convergence of gradient descent]\label{th:convergencedescent}
Under {\sf (A1-5)}, let $J_0$ and $\kappa_0$ be given by Theorem~\ref{th:gradientinequality} and  $\rho = \rho\!\vol \in \Mm_+(\Theta)$ an absolutely continuous reference measure with $\log \rho$ Lipschitz. For any $0 <\epsilon\leq 1/2$ and $\nu_0\in \Mm_+(\Theta)$, there exists $C,C'>0$ that depends on the characteristics of the problem and increasingly on $\bar \Hh(\nu^\star,\nu_0)$ and $1/\epsilon$, such that if 
\begin{align*}
W_\infty(\nu_0,\rho)\leq (J_0-J^\star)/C, && \alpha \leq (J_0 - J^*)^{1+\epsilon/2}/C, &&\text{and}&& \beta \leq (J_0 - J^*)\alpha^2/C'
\end{align*}
 then the projected gradient descent $(\nu_k)_{k\in \NN}$ initialized with $\nu_0$  converges to the global optimum $\nu^\star$. Denoting $k_0 = C/(J_0 - J^*)^{2+\epsilon}$ it satisfies, for $k\geq k_0$,
\begin{align*}
  J(\nu_k)-J^\star \leq (J(\nu_0) - J^\star) \left(1-\kappa_0\cdot\min\{\alpha,\beta\} \right)^{k-k_0}.
  \end{align*}
  \end{theorem}
We can make the following comments:
\begin{itemize}
\item[--] The non-asymptotic convergence rate does not appear explicitly in Theorem~\ref{thm:convergenceflow}, because the result is obtained by trading-off various error terms. In an the idealized setting where $\nu_0=\rho$ and $\beta=0$, a direct consequence of Lemma~\ref{lem:mirrordecrease} and Lemma~\ref{lem:qdensity} is that $J(\nu_t)-J^\star$ decreases as $O(\log(t)/t)$ for the gradient flow and in $O(\log(k)/\sqrt{k})$ for the gradient descent in general. For the specific case of the mirror retraction, we show  in Appendix~\ref{app:fasterrate} that a faster rate in $O(\log(k)/k)$ holds.
\item[--] The condition on the initialization can be achieved by taking $\nu_0$ a weighted empirical distribution of $m$ samples from $\rho$ (typically the normalized volume measure), and it is known that the rate of convergence in $W_\infty$ of such approximation is in~$\tilde O(m^{-1/d})$, see~\cite{trillos2015rate}. Unfortunately, this exponential dependence in the dimension is unavoidable when approximating densities in Wasserstein distances~\cite{weed2017sharp}. This corresponds to a quantitative version of the condition on the initialization in Theorem~\ref{th:consistency}. Also, note that $J_0-J^\star$ gets smaller as the problem becomes more difficult, in which case the overparameterization $m$ must increase, and the convergence speed slows down. In particular, the necessary condition $m\geq m^\star$ is implicitly implied by our assumptions.
\item[--] The fact that the sublevel $J_0$ from Theorem~\ref{th:gradientinequality} does not depend on the metric parameters $(\alpha,\beta)$ is crucial to prove these theorems. However, the local exponential rate of convergence in Theorem~\ref{th:convergencedescent} may be deceptively bad if $\beta/\alpha$ is extremely small. An natural fix is to start with a small ratio $\beta/\alpha$ as required by Theorem~\ref{th:convergencedescent}, and to increase this ratio at each iteration so as to improve the conditioning of $J$ near optimality. The interest of Theorem~\ref{th:convergencedescent} lies mostly in the qualitative insights it brings. In practice, we would advise to choose $W_\infty(\nu_0,\rho)$, $\alpha$ and $\beta$ via heuristics or parameter search rather than trying to derive the constants of Theorem~\ref{th:convergencedescent}, which could be deceptively conservative.
\end{itemize}

\subsection{Proof of global convergence for gradient flows}\label{sec:proofglobal}
 The proof of Theorem~\ref{thm:convergenceflow} mostly relies on the the following general lemma which applies to any type of initialization or any structure of minimizers. It gives an upper bound on the optimality gap during along gradient flows in terms of a \emph{mirror rate} function $\Qq_{\nu_0,\nu^\star}: \RR^*_+ \to \RR_+$ defined for $\nu^\star, \nu_0\in \Mm_+(\Theta)$  and $\tau>0$ as
\begin{align}\label{eq:quantization}
\Qq_{\nu^\star,\nu_0}(\tau) = 
\inf_{\nu \in \Mm_+(\Theta)} \Vert \nu^\star- \nu\Vert_{\BL}^* +\frac{1}{\tau} \Hh(\nu,\nu_0).
\end{align}
This is a continuous and decreasing function of $\tau$ that satisfies
\[
 \lim_{\tau \to \infty} \Qq_{\nu^\star,\nu_0}(\tau)  = \inf_{\spt \nu\subset \spt \nu_0} \Vert \nu^\star- \nu\Vert_{\BL}^*
\]
which is $0$ if and only if $\spt(\nu^\star)\subset \spt(\nu_0)$. When $\beta=0$, this function directly controls the rate of convergence of this mirror descent dynamics hence the name \emph{mirror rate} function.

\begin{lemma}\label{lem:mirrordecrease}
Assume ${\sf (A1-3)}$ and that $J$ admits a minimizer $\nu^\star\in \Mm_+(\Theta)$. Then for all $\nu_0\in \Mm_+(\Theta)$, denoting $B_{\nu_0}\coloneqq \sup_{J(\nu)\leq J(\nu_0)} \Vert J'_{\nu}\Vert_{\BL}$, it holds for $t\geq 0$,
\begin{align*}
J(\nu_t)-J^\star \leq \inf_{s\in [0,t]}\left( B_{\nu_0} \cdot \Qq_{\nu^\star,\nu_0}(4\alpha B_{\nu_0}s) + \beta B^2_{\nu_0} s\right).
\end{align*}
\end{lemma}
A direct consequence of this lemma is that $\lim_{t\to \infty} J(\nu_t)-J^\star $ is guaranteed to be small as $\beta$ gets smaller and as  $\spt \nu_0$ gets closer to $\spt \nu^\star$. In Appendix~\ref{app:mirror} we give an upper bound on $\Qq$ for the situation of interest here, leading to explicit convergence rates when combined with Lemma~\ref{lem:mirrordecrease}.
\begin{proof}
Let $\nu^\epsilon_0 \in \Mm_+(\Theta)$ be a measure to be specified later that satisfies $\Hh(\nu^\epsilon_0, \nu_0)<+\infty$, and let $\nu^\epsilon_t$ satisfy $\partial_t \nu^\epsilon_t = \div(\beta \nu^\epsilon_t\nabla J'_{\nu_t})$ for $t\geq 0$ weakly (this is a continuity equation with a smooth velocity field which admits a unique weak solution). Differentiating the relative entropy with respect to its second argument and using the invariance of the relative entropy under diffeomorphisms, it holds, for $t\geq 0$,
\begin{align*}
\frac1{4\alpha} \frac{d}{dt} \Hh(\nu^\epsilon_t, \nu_t) &= \int_{\Theta} J'_{\nu_t} \d (\nu^\epsilon_t-\nu_t) \\
& = \int_{\Theta} J'_{\nu_t} \d(\nu^\star -\nu_t) + \int_{\Theta} J'_{\nu_t} \d(\nu^\epsilon_t-\nu^\star) \\
& \leq -(J(\nu_t)-J^\star) + \Vert J'_{\nu_t}\Vert_\BL \cdot \Vert \nu^\star-\nu^\epsilon_t\Vert_\BL^*
\end{align*}
where the first term comes from the convexity of $J$ and the second from the definition of $\Vert \cdot\Vert_{\BL}^*$. After integrating in time and rearranging the terms we get
\begin{align*}
\left(\frac1t\int_0^t J(\nu_s)\d s\right)-J^\star &\leq \frac1{4\alpha t} \left( \Hh(\nu^\epsilon_0,\nu_0) - \Hh(\nu^\epsilon_t,\nu_t) \right) + \frac{B_{\nu_0}}{t} \int_0^t \Vert \nu^\star-\nu^\epsilon_s\Vert_{\BL}^* \d s.
\end{align*}
For the last integral term, we use the triangular inequality
\[
 \frac{B_{\nu_0}}{t}  \int_0^t \Vert \nu^\star-\nu^\epsilon_s\Vert_{\BL}^* \d s \leq B_{\nu_0}\Vert \nu^\star - \nu^\epsilon_0\Vert_{\BL}^* + \frac{B_{\nu_0}}{t} \int_0^t \Vert \nu_0^\epsilon -\nu^\epsilon_s\Vert_{\BL}^*\d s\leq B_{\nu_0} \Vert \nu^\star - \nu^\epsilon_0\Vert_{\BL}^* +B_{\nu_0}^2\beta t
\]
where the last term is obtained by bounding the integrated flow of the velocity field $(\nabla J'_{\nu_t})_{t\geq 0}$. Since $\Hh(\nu_\epsilon,\nu_t)\geq 0$ and $J(\nu_s)$ is decreasing, it follows
\[
J(\nu_t)-J^\star  \leq \inf_{\nu^\epsilon\in \Mm_+(\Theta)} \left( \frac1{4\alpha t} \Hh(\nu^\epsilon,\nu_0) + B_{\nu_0}\Vert  \nu^\star -\nu^\epsilon\Vert_{\BL}^*\right) + B_{\nu_0}^2 \beta t.\qedhere
\]
\end{proof}

\begin{proof}[Proof of Theorem~\ref{thm:convergenceflow} (gradient flow)]
By Lemma~\ref{lem:qdensity}, we have for $\tau\geq L=\Lip(\log \rho)$, by writing $\bar \Hh \coloneqq \bar \Hh(\nu^\star,\rho\vol)$,
\[
\Qq_{\nu^\star,\hat \nu_0}(\tau) \leq \frac{\bar \Hh+d\cdot \nu^\star(\Theta)\cdot (\log(\tau)+C_\Theta)}\tau + \nu^\star(\Theta)\cdot W_\infty(\nu_0,\rho\vol).
\]
Combining this bound with Lemma~\ref{lem:mirrordecrease}, we get that for $t\geq L/(4\alpha B_{\nu_0})$,
\[
J(\nu_t)-J^\star \leq \frac{\bar \Hh+d\cdot\nu^\star(\Theta)\cdot(\log(4\alpha B_{\nu_0} t)+C_\Theta)}{4\alpha B_{\nu_0} t} + B_{\nu_0}\cdot\nu^\star(\Theta)\cdot W_\infty(\nu_0,\rho\vol)+\beta B^2_{\nu_0} t.
\]
In particular, for $t=(\alpha\beta)^{-\frac12}$, we get 
\begin{equation}\label{eq:lastbeforelast}
J(\nu_t)-J^\star \leq \sqrt{\frac{\beta}{\alpha}}\left(\bar \Hh + d\cdot \nu^\star(\Theta) (\log(4B_{\nu_0}\sqrt{\alpha/\beta})+C_\Theta) + B^2_{\nu_0}\right) + B_{\nu_0}\nu^\star(\Theta)\cdot W_\infty(\nu_0,\rho\vol).
\end{equation}
Since this is valid only when $t\geq L/(4\alpha B_{\nu_0})$, we require $(\alpha\beta)^{-\frac12}\geq L/(4\alpha B_{\nu_0})$ which leads to the first condition on $\beta/\alpha$. Now, we want the right-hand side of~\eqref{eq:lastbeforelast} to be smaller than $\Delta_0 \coloneqq J_0 - J^\star$ so that we can conclude with Corollary~\ref{cor:localGF}. To this end, we require, on the one hand $W_\infty(\nu_0,\rho\vol) \leq \Delta_0/(2B_{\nu_0}\nu^\star(\Theta))$. On the other hand, we use the bound $\log(u)\leq C_\epsilon u^\epsilon$ for $\epsilon\in {]0,1/2]}$, require $4B_{\nu_0}\sqrt{\alpha/\beta}\geq 1$ and obtain the condition
\[
(4B_{\nu_0})^\epsilon(\beta/\alpha)^{(1-\epsilon)/2}\left( \bar \Hh + d\cdot \nu^\star(\Theta)(C_\Theta+C_\epsilon) +B^2_{\nu_0}\right)\leq \frac12\Delta_0.
\]
This leads to the second condition on $\beta/\alpha$ is the theorem.
\end{proof}

\subsection{Fully non-convex gradient descent}\label{sec:fullynonconvex}
The results in the previous section require to set  $\beta/\alpha$ at a small initial value. This might appear undesirable because the asymptotic convergence result of Theorem~\ref{th:consistency} holds irrespective of the choice of $\beta/\alpha$. Also, in practice, this condition does not seem required, at least in the examples that we have considered (see Section~\ref{sec:numerics}).  While the proof technique from Section~\ref{sec:proofglobal} fails without controlling $\beta/\alpha$, the question of wether it is possible to obtain convergence rates for any ratio $\beta/\alpha$ is a natural one.

For such a result, the key challenge is to obtain a convergence rate for the gradient flow dynamics~\eqref{eq:projectedGF} when initialized with a positive density, without conditions on $(\alpha,\beta)$. While we were not able to prove such a result, in order to point out at the theoretical difficulty, we show in Appendix~\ref{sec:boundeddensity} with a proof technique inspired by~\cite{wei2019regularization}, that a convergence rate in objective value in $O(1/\sqrt{\eta t})$ holds as long as the density $\nu_t$ is lower bounded by some $\eta>0$ (at least on a certain subset of $\Theta$).

\begin{proposition}
Under {\sf (A1-3)}, for any $J_{\max}\geq J^\star$, there exists $C>0$ such that for any $\eta,t>0$ and $\nu_0\in \Mm_+(\Theta)$ satisfying $J(\nu_0)\leq J_{\max}$, if the projected gradient flow~\eqref{eq:projectedGF} satisfies for $0\leq s \leq t$,
\[
\nu_s\vert_{S_t} \geq \eta \vol\vert_{S_t}
\]
where $S_t = \{\theta \in \Theta\;;\; J'_{\nu_s}(\theta)\leq 0 \text{ for some } s\in {[0,t]}\}$, then 
$
J(\nu_t)-J^\star\leq \frac{C}{\sqrt{\alpha \eta t}}.
$
\end{proposition}

Unfortunately, this result is not sufficient to obtain a convergence rate because the lower bound on the density may decrease too fast. When this happens, the gradient flow may stagnate an \emph{a priori} unbounded time in neighborhoods of saddle points, although it is guaranteed to eventually escape by Lemma~\ref{lem:escape}. Note that the result above does not requires $\Ff$ to be finite dimensional nor $\lambda=0$ while this would be needed for a proof based on the positive definiteness of the tangent kernel~\cite{jacot2018neural}.

\section{Numerical experiments}\label{sec:numerics}
All experiments can be reproduced with the Julia code available online\footnote{\url{https://github.com/lchizat/2019-sparse-optim-measures}}. Our goal here is not to demonstrate the superiority of Algorithm~\ref{alg:gradientbased} over other algorithms, but rather to illustrate the insights obtained by the analysis. We consider the following problems introduced in Section~\ref{sec:applications} :
\begin{itemize}
\item[--] (Sparse deconvolution) We consider the Dirichlet low-pass filter of order $n_f\in \NN_*$ on the $d$-torus with values in $L^2(\TT^d)$ i.e.\  $\phi(\theta): x \mapsto \sum_{k=-n_f}^{n_f}\exp(k\sqrt{-1}(x-\theta))$ when $d=1$. We use the square-loss and solve problem~\eqref{eq:objective} with conic particle gradient descent (Algorithm~\ref{alg:gradientbased}) with the ``mirror retraction'' from Section~\ref{sec:gradientdescent}. 
\item[--] (Two-layer neural net) We consider  the function $\phi(w): x \mapsto \max\left\{\sum_{j=1}^{d} x_j w_j\cdot \vert w_j\vert, 0\right\}$ which is $2$-homogeneous on $\RR^{d+1}$ with $d+1=20$. We use the square loss and solve problem~\eqref{eq:objective} with stochastic gradient descent with a small fixed step-size for an input data distribution uniform on the sphere $\SS^d$. This corresponds to a stochastic version of Algorithm~\ref{alg:gradientbased} with the ``induced retraction'' from Section~\ref{sec:gradientdescent}. For our purposes, the advantage of this architecture over classical ReLU neural networks (as presented in Section~\ref{sec:applications}) is that here $\phi$ is differentiable on $\SS^d$ (see, e.g.~\cite[Lem. D.5]{chizat2018global}).
\end{itemize}
We focus in both cases on the ``teacher-student'' setting without noise with the square loss, because it guarantees that even the unregularized problem ($\lambda=0$) has sparse solutions, in spite of  $\Ff$ being infinite dimensional. We thus have $R(f) = \frac12 \Vert f -f^\star\Vert^2_{\Ff}$ where $f^\star = \sum_{i=1}^{m_0} r_i^2 \phi(\theta_i)$ and $m_0\in \NN^*$ is the number of atoms for the teacher.

\begin{figure}
\centering
\begin{subfigure}{0.33\linewidth}
\centering
\includegraphics[scale=0.45]{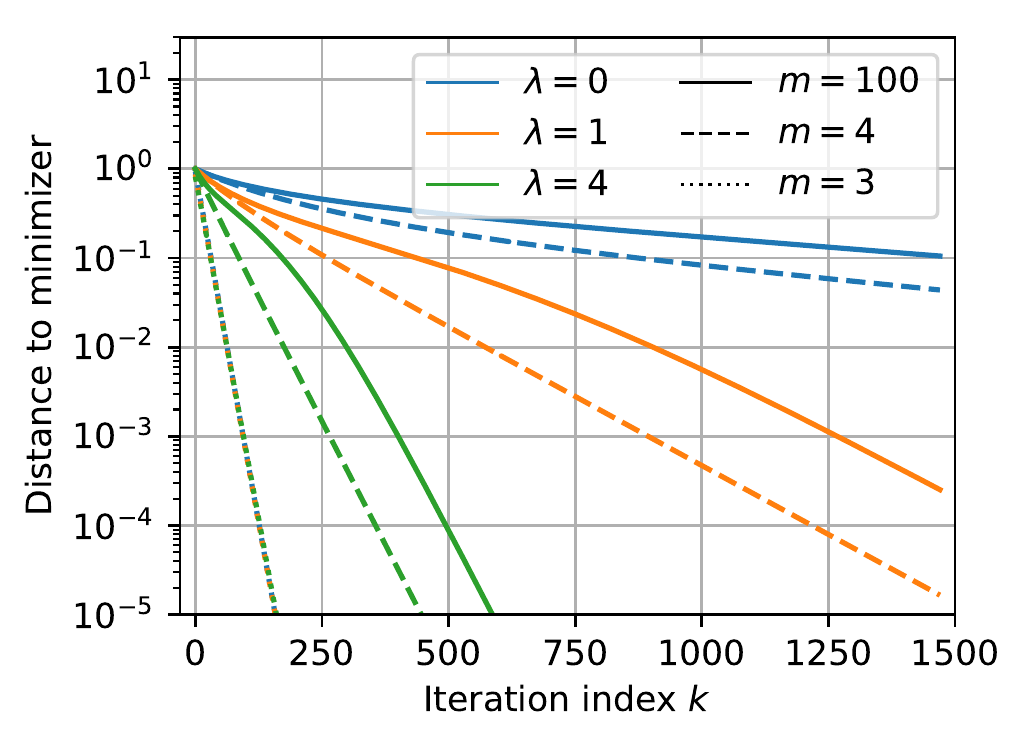}
\caption{Deconvolution 2D (distance)}
\end{subfigure}%
\begin{subfigure}{0.33\linewidth}
\centering
\includegraphics[scale=0.45]{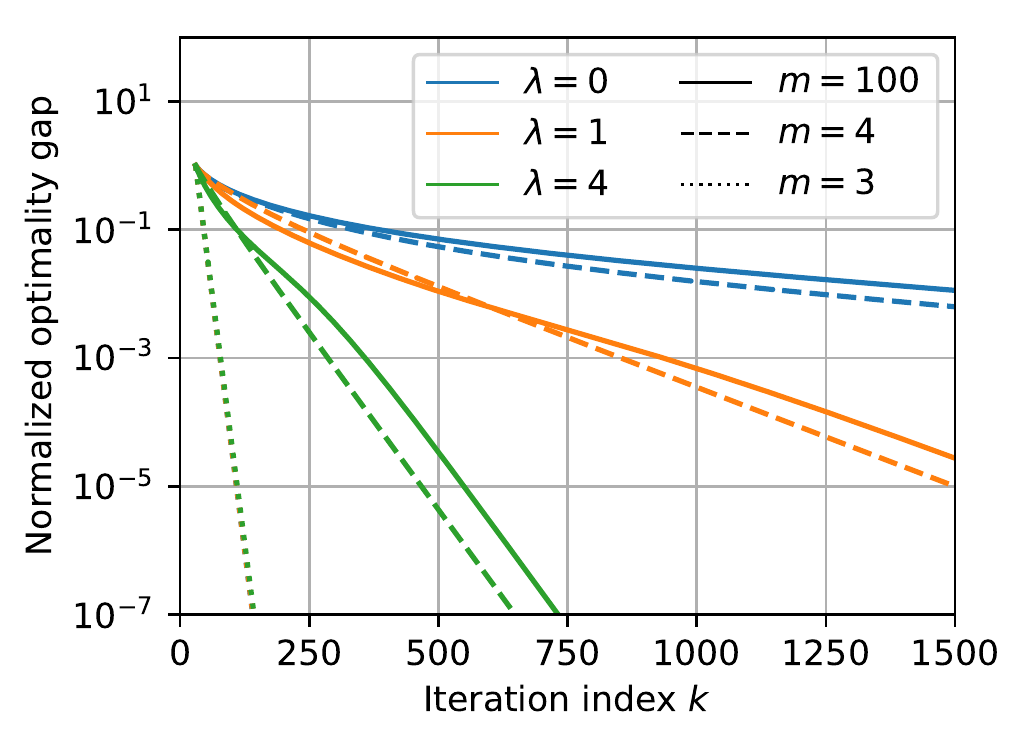}
\caption{Deconvolution 2D (loss)}
\end{subfigure}%
\begin{subfigure}{0.33\linewidth}
\centering
\includegraphics[scale=0.45]{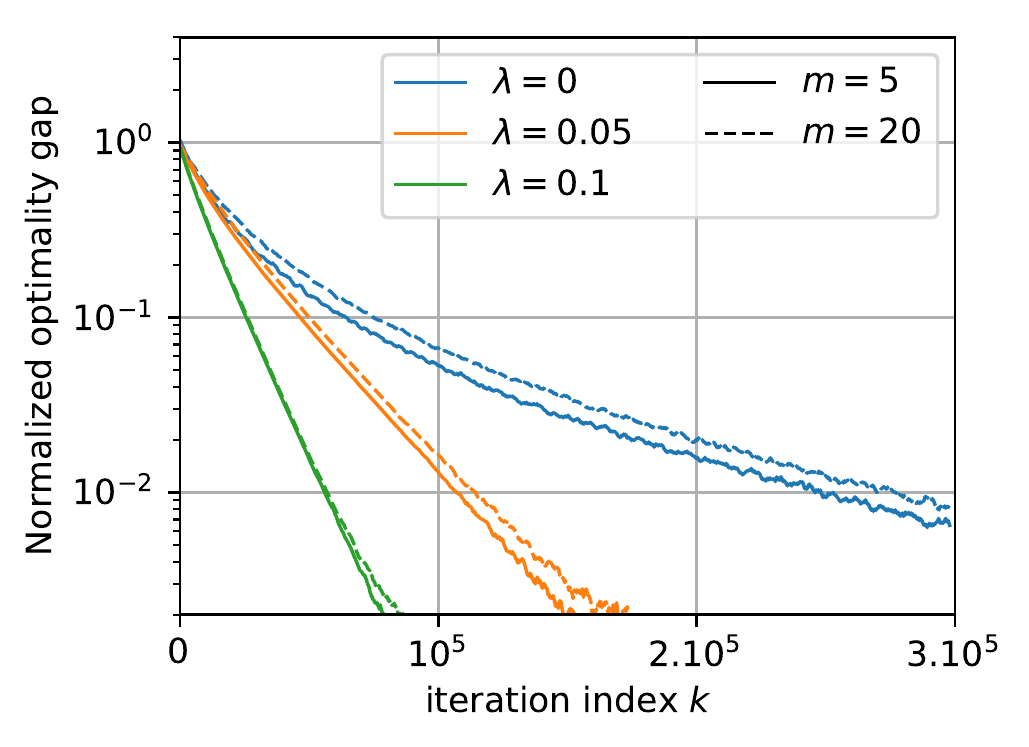}
\caption{Two-layer neural net (loss)}
\end{subfigure}%
\caption{Convergence plots (normalized to $1$ after a burning period of $30$ iterations). The initialization is close to the minimizer for deconvolution, and is random for neural nets.}
\label{fig:rates}
\end{figure}

\paragraph{Local convergence rate.} We observe on Figure~\ref{fig:rates} the effect of the regularization parameter $\lambda$ and of the over-parameterization parameter $m$ on the local convergence rates (in $\widehat W_2$ distance -- approximated by mapping each particle to its final position/mass -- or in optimality gap).  In accordance with the expansion of Proposition~\ref{prop:expansions}, we observe exponential convergence whenever $\lambda>0$, with a rate that improves as $\lambda$ increases. For sparse deconvolution, we observe fast exponential convergence when $m=m_0=3$ which is explained by only the first term in the local expansion~\eqref{eq:expansion} being non-zero. By adding just a single particle, the second term comes into play and the behavior is qualitatively similar than with $20$ particles. For Figure~\ref{fig:rates}-(c), the initialization is random and $m_0=5$. Here the behavior for $m=m_0$ follows that of $m>m_0$ which suggests that the first term in the local expansion of Eq.~\eqref{eq:expansion} dominates.

\paragraph{Global convergence.} We observe on Figure~\ref{fig:success} the effect on the success/failure of optimization of the two main parameters that appear in Theorem~\ref{thm:convergenceflow}:  the over-parameterization parameter $m$ (used to decrease the $W_\infty$ criterion) and the ratio of the vertical/spatial step-sizes $\beta/\alpha$. In both~(a) and (b) we have $m_0=5$ and $\lambda>0$, and the final loss is averaged over $5$ random experiments. Without surprise, minimizers cannot be reached when $m$ is too small. It is also observed that increasing $m$ increases the chances of success even when $m\geq m_0$. In contrast, these experiments do not reveal a clear role for $\beta/\alpha$, beyond a change in the convergence speed (see Section~\ref{sec:fullynonconvex}).

\begin{figure}
\centering
\begin{subfigure}{0.3\linewidth}
\centering
\includegraphics[scale=0.5]{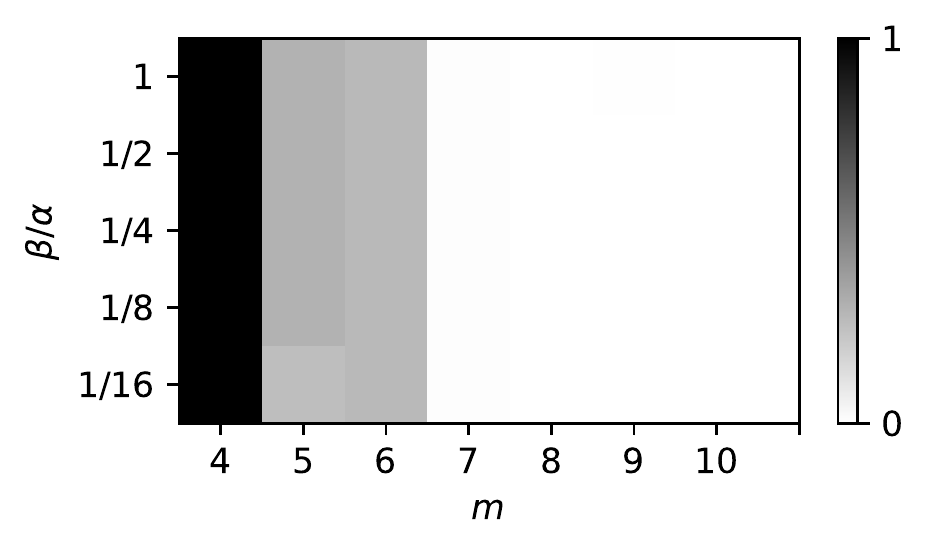}
\caption{Deconvolution 1D}
\end{subfigure}%
\hspace{1cm}
\begin{subfigure}{0.3\linewidth}
\centering
\includegraphics[scale=0.5]{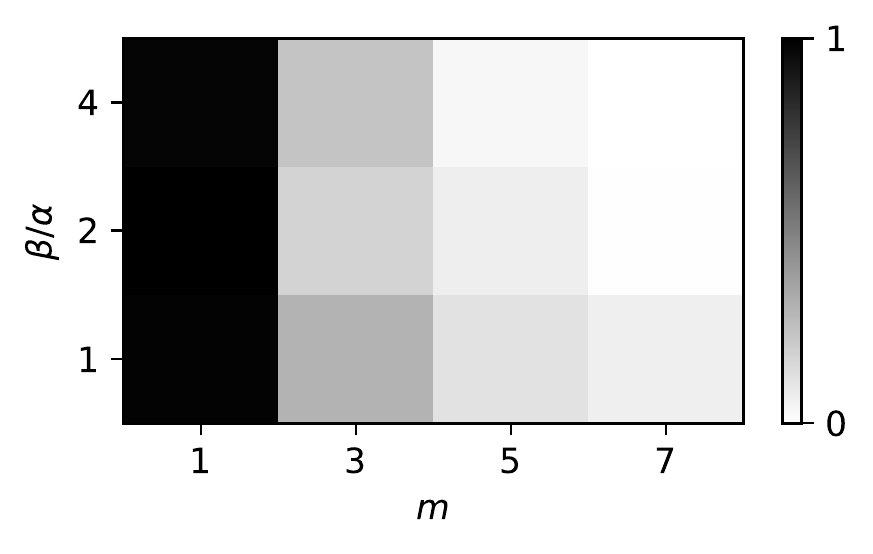}
\caption{Two-layer neural net}
\end{subfigure}%
\caption{Effect of $m$ and $\beta/\alpha$ on the excess loss after a fixed large number of iterations ($\lambda$ fixed). The shades go from white (lowest objective achieved) to black (highest).}
\label{fig:success}
\end{figure}

\paragraph{Comparison of vertical geometries.} Finally, we compare on Figure~\ref{fig:convex} the behavior of mirror descent against that of Euclidean descent (here integrated with ISTA algorithm~\cite{daubechies2004iterative}). This corresponds respectively to $h(r)=r^2$ and $h(r)=r$ in Eq.~\ref{eq:objectivediscrete} and $\beta=0$. We consider the problem of recovering a single spike ($m_0=1$) for 1D and 2D sparse deconvolution, starting from the uniform measure on $\Theta$ densely sampled on a grid ($m=100$). We report  the behavior in early stages of optimization, before the effect of the discretization comes into play. We observe that mirror descent outperforms Euclidean descent and enjoys a convergence rate of order $\sim 1/k$ around iteration number $k=100$. This is in accordance with the result of Appendix~\ref{app:fasterrate}, where we show a convergence rate for mirror descent with continuous densities in $O(\log(k)/k)$, independent of the dimension. The difference in behavior is illustrated on Figure~\ref{fig:convex}-(c) where we plot $\nu_{1000}$ (in the setting of panel (a)).

\begin{figure}
\centering
\begin{subfigure}{0.33\linewidth}
\centering
\includegraphics[scale=0.44]{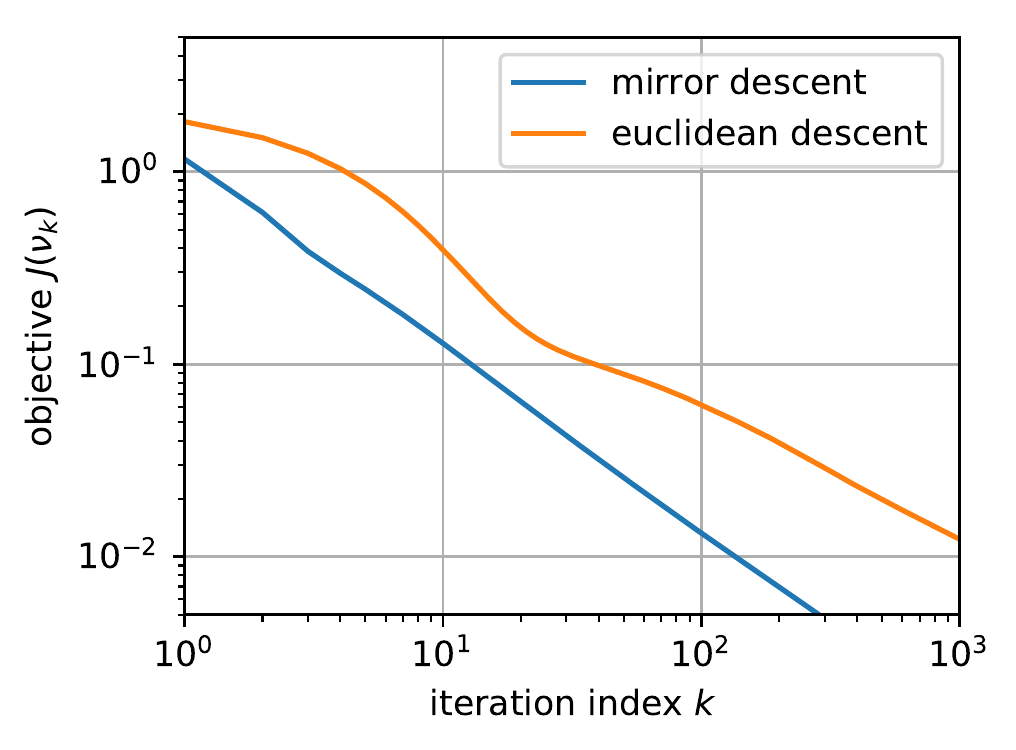}
\caption{Deconvolution 1D}
\end{subfigure}%
\begin{subfigure}{0.33\linewidth}
\centering
\includegraphics[scale=0.44]{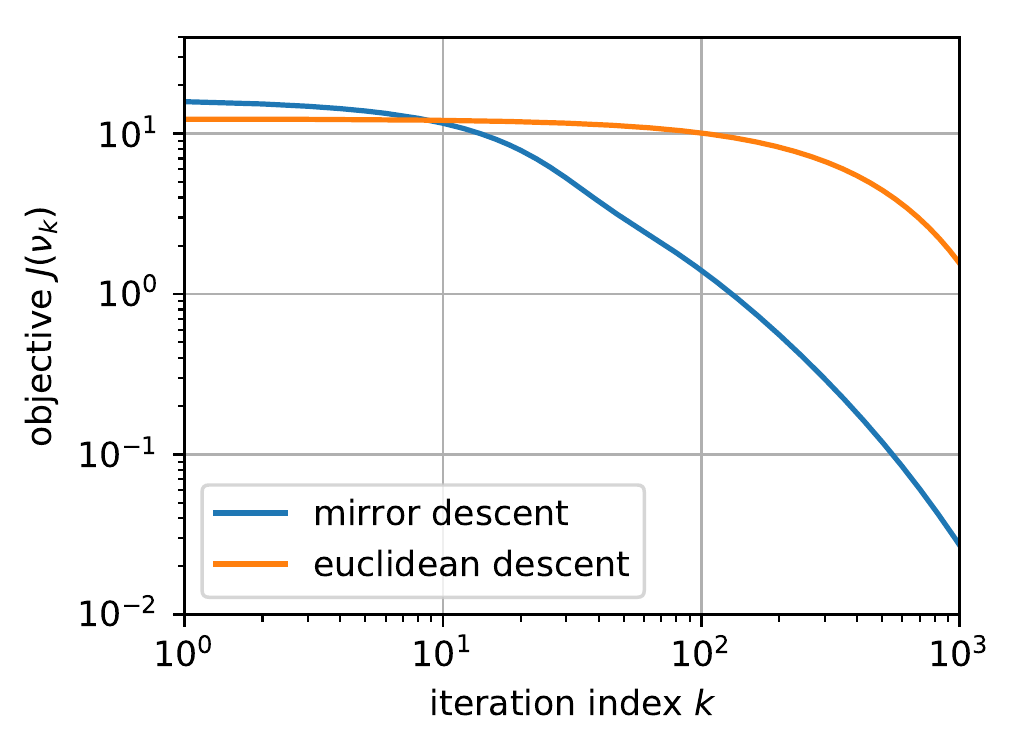}
\caption{Deconvolution 2D}
\end{subfigure}%
\begin{subfigure}{0.33\linewidth}
\centering
\includegraphics[scale=0.45]{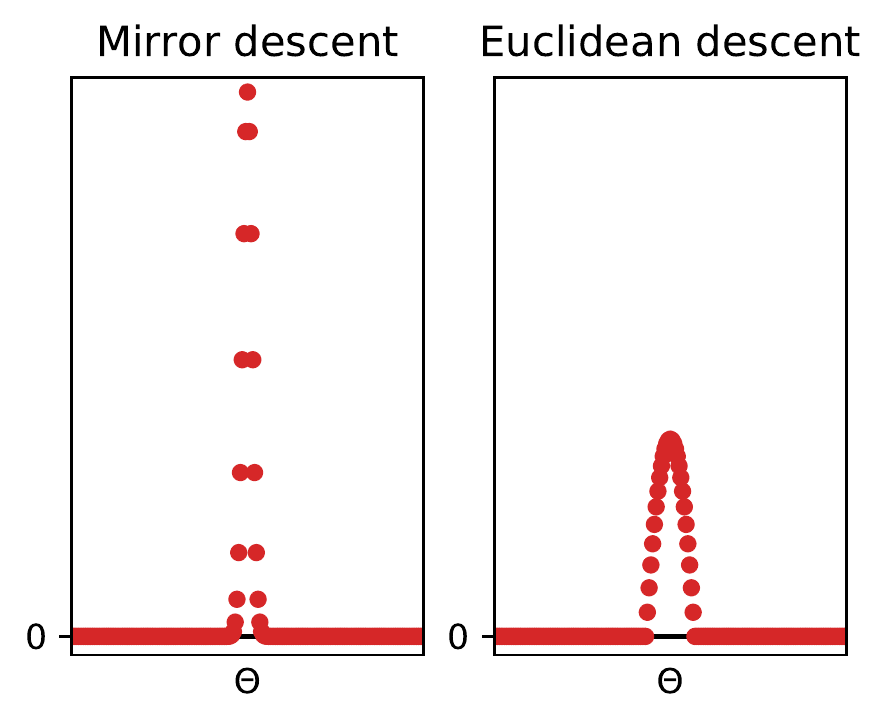}
\caption{Some iterate $\nu_k$}
\end{subfigure}%
\caption{Comparison of mirror ($h(r)=r^2$) and Euclidean ($h(r)=r$) dynamics when $\beta=0$. }
\label{fig:convex}
\end{figure}

\section{Conclusion}
In this paper, we have studied particle gradient descent for sparse convex optimization on measures and obtained complexity guarantees under non-degeneracy assumptions. One central idea underlying our analysis is to directly study the iterates in Wasserstein space. We believe that this approach, at the crossroads between analysis and optimization, may lead to other insights for over-parameterized and non-convex gradient descent.

An avenue for future research is to study the unregularized case. This may require to exploit finer properties of the problem than mere smoothness and could improve our understanding of the implicit bias of over-parameterized gradient descent. Another important question is to find theoretical explanations for the favorable behavior observed in high dimensions for two layer neural networks optimization.

\subsubsection*{Acknowledgments} The author thanks Francis Bach for fruitful discussions related to this work and the anonymous referees for their thorough reading and suggestions.

\bibliography{LC}
\bibliographystyle{plain}

\newpage
\appendix

\section{Dealing with signed measures}\label{app:signed}
Let us show that problems over signed measures with total variation regularization are covered by problem~\eqref{eq:objective}, after a suitable reformulation. Consider a function $\tilde \phi:\tilde \Theta\to \Ff$ and the functional on signed measures $\tilde J: \Mm(\tilde \Theta)\to \RR$ defined as
\begin{equation}\label{eq:signedobjective}
\tilde J(\mu) = R\left(\int \tilde \phi \d\mu\right) + \lambda \vert \mu \vert(\tilde \Theta).
\end{equation}
where $\vert \mu \vert(\tilde \Theta)$ is the total variation of $\mu$. This is a continuous version of the LASSO problem, known as BLASSO~\cite{de2012exact}. Define $\Theta$ as the disjoint union of two copies $\tilde \Theta_+$ and $\tilde \Theta_-$  of $\tilde \Theta$ and define the symmetrized function $\phi:\Theta\to \Ff$ as 
\[
\phi(\theta) = \begin{cases}
+\tilde \phi(\theta) &\text{if $\theta \in \tilde \Theta_+$}\\
-\tilde \phi(\theta) &\text{if $\theta \in \tilde \Theta_-$}
\end{cases}.
\]
With this choice of $\phi$, minimizing~\eqref{eq:signedobjective} or minimizing~\eqref{eq:objective} are equivalent, in a sense made precise in Proposition~\ref{prop:signed}. This symmetrization procedure, also suggested in~\cite{chizat2018global}, is simple to implement in practice: in Algorithm~\ref{alg:gradientbased}, we fix at initialization the sign attributed to each particle --- depending on whether it belongs to $\tilde \Theta_+$ or $\tilde \Theta_-$ --- and do not change it throughout the iterations.
\begin{proposition} \label{prop:signed}
The infima of~\eqref{eq:signedobjective} and~\eqref{eq:objective} are the same and:
\begin{enumerate}[(i)]
\item  if $\tilde \mu$ is a minimizer of $\tilde J$ and $\tilde \mu = \tilde \mu_+-\tilde \mu_-$ is its Jordan decomposition, then the measure which restriction to $\tilde \Theta_+$ (resp.\ $\tilde \Theta_-$) coincides with $\tilde \mu_+$ (resp.\ $\mu_-$) is a minimizer of $J$;\label{prop:signedA}
\item reciprocally, if $\mu$ is a minimizer of $J$ then $\mu_+-\mu_-$ where $\mu_+$ (resp. $\mu_-$) is the restriction of $\mu$ to $\tilde \Theta_+$ (resp.\ $\Theta_-$) is a minimizer of $\tilde J$. \label{prop:signedB}
\end{enumerate}
\end{proposition}
\begin{proof}
We recall that for any decomposition of a signed measure as the difference of nonnegative measures $\tilde \mu = \tilde \mu_+-\tilde \mu_-$, it holds $\vert \tilde\mu \vert(\Theta) \leq \tilde\mu_+(\Theta)+\tilde\mu_-(\Theta)$, with equality if and only if $(\tilde \mu_+,\tilde \mu_-)$ is the Jordan decomposition of $\tilde \mu$~\cite[Sec. 4.1]{cohn1980measure}.
It follows that starting from any $\tilde \mu\in \Mm(\tilde \Theta)$, the construction in~\eqref{prop:signedA} yields a measure $\mu \in \Mm_+(\Theta)$ satisfying $\tilde J(\tilde \mu) = J(\mu)$. Also, starting from any $\mu \in \Mm_+(\Theta)$, the construction in~\eqref{prop:signedB}  yields a measure $\tilde \mu\in \Mm(\tilde \Theta)$ satisfying $\tilde J(\tilde \mu)\leq J(\mu)$, with equality if and only if $(\mu_+,\mu_-)$ is a Jordan decomposition. The conclusion follows.
\end{proof}

\section{Generic non-convex minimization}\label{app:genericnonconvex}
In this section, we show that any smooth optimization problem on a manifold is equivalent to solving a problem of the form~\eqref{eq:objective}. This corresponds to the case of a scalar-valued $\phi$.
\begin{proposition}
Let $\phi:\Theta \to \RR$ be a smooth function with minimum $\phi^\star<0$ that admits a global minimizer, and let
\begin{align}\label{eq:genericoptim}
\nu^\star \in \arg\min_{\nu \in \Mm_+(\Theta)} J(\nu)&& \text{where}
&&
J(\nu)\coloneqq \frac12\left( 2 + \int_\Theta \phi(\theta)\d \nu(\theta) \right)^2 + \lambda \nu(\Theta)
\end{align}
where $0<\lambda<-2\phi^\star$. Then $\emptyset \neq \spt \nu^\star \subset \arg\min \phi$ so minimizers of $\phi$ can be built from $\nu^\star$. Reciprocally, from a minimizer of $\phi$, one can build a minimizer for~\eqref{eq:genericoptim}.
\end{proposition}
\begin{proof}
For a measure $\nu\in \Mm_+(\Theta)$, we define $f_\nu \coloneqq \int_\Theta \phi(\theta)\d\nu(\theta) \in \RR$. It holds
\[
\int_{\Theta}J'_\nu(\theta)\d\nu(\theta) = \int_{\Theta}\left( \phi(\theta)(2+f_\nu) + \lambda \right)\d\nu(\theta)= f_\nu^2 +2f_\nu + \lambda \nu(\Theta).
\]
Now suppose that $\nu$ is a global minimizer of $J$. Then the optimality condition in Proposition~\ref{prop:optimalitycondition} implies that 
\begin{equation}\label{eq:generalnonconvexpoly}
    f_\nu^2 +2f_\nu + \lambda \nu(\Theta)=0.
\end{equation} 
Solving for $f_\nu$ is possible if $\lambda \nu(\Theta)<1$ and leads to $f_\nu =\sqrt{1-\lambda \nu(\Theta)} -1$. We also deduce from the fact that $f_\nu>-1$ that $\arg\min J'_\nu = \arg\min \phi$, and so $\spt \nu \subset \arg\min \phi$. It remains to find under which condition $\nu(\Theta)>0$.
We use the fact that $f_\nu = \phi^\star\nu(\Theta)$ in Equation~\eqref{eq:generalnonconvexpoly}, and get
\[
\nu(\Theta) = \max\left\{ 0, \frac{-2\phi^\star-\lambda}{(\phi^\star)^2} \right\}
\]
which in particular satisfies $\lambda \nu(\Theta)<1$. Thus, as long as $-2\phi^\star>\lambda$, we have $\nu(\Theta)>0$. Finally, we verify that global minimizers exist, so that the above reasoning makes sense. If $-2\phi^\star -\lambda\leq 0$, then $\nu=0$ satisfies the global optimality conditions. Otherwise, choose $\theta^\star$ a minimizer  for $\phi^\star$ and define $\nu = \nu(\Theta)\delta_{\theta^\star}$ with the value above for $\nu(\Theta)$, which also satisfies the global optimality conditions.
\end{proof}

\section{Wasserstein gradient flow}
In this section, we recall and adapt some results and proofs from~\cite{chizat2018global}, for the sake of completeness.

\subsection{Existence}\label{app:existence}
For this result, we assume {\sf (A1-2)}. 
For a compactly supported initial condition $\mu_0\in \Pp_2(\Omega)$, the proof of existence for Wasserstein gradient flows (Eq.~\eqref{eq:WGF}) in~\cite{chizat2018global} goes through, as it is simply based on a compactness arguments which can be directly translated to this Riemannian setting (more precisely, we apply here Arzel\`a-Ascoli compactness criterion for curves in the Wasserstein space on the cone of $\Theta$, which is a complete metric space~\cite{liero2018optimal}). Note that these arguments do not require convexity of $R$, but in order to guarantee global existence in time, we need to assume that $\nabla R$ is bounded in sub-level sets of $F$. 

For the existence of solutions for projected dynamics on $\Theta$ for any $\nu_0\in \Mm_+(\Theta)$, consider a measure $\mu_0\in \Mm_+(\Omega)$ such that $\h\mu_0=\nu_0$ (see~\cite{liero2018optimal} for such a construction) and the corresponding Wasserstein gradient flow $(\mu_t)_{t\geq 0}$ for $F$. Then $\h\mu_t$ is a solution to~\eqref{eq:projectedGF}.

For the existence of Wasserstein gradient flows (Eq.~\eqref{eq:WGF}) for $F$ when $\mu_0$ is not compactly supported, proceed as follows: there exists a Wasserstein-Fisher-Rao gradient flow $\nu_t$ satisfying $\nu_0=\h\mu_0$. Now we can simply define $\mu_t$ as the solution to $\partial_t \mu_t = \div(\mu_t J'_{\nu_t})$. It can be directly checked that $\h\mu_t = \nu_t$ for $t\geq 0$ and thus $\mu_t$ is a solution to Eq.~\eqref{eq:WGF}.

We do not attempt to show uniqueness in the present work. Note that it is proved in~\cite{chizat2018global} for the case where $\Theta$ is a sphere, by applying the theory developed in~\cite{ambrosio2008gradient}.

\subsection{Asymptotic global convergence}
In this section, we give a short proof of Theorem~\ref{th:consistency}, adapted from~\cite{chizat2018global}. The next lemma is the crux of the global convergence proof. It gives a criterion to espace from the neighborhood of measures which are not minimizers.

\begin{lemma}[Criteria to espace local minima]\label{lem:escape}
Under {\sf (A1-3)}, let $\nu \in \Mm_+(\Theta)$ be such that  $v^\star \coloneqq \min_{\theta \in \Theta} J'_{\nu}(\theta)<0$. Then there exists $v \in [2v^\star/3,v^\star/3]$ and  $\epsilon>0$ such that   if $(\nu_t)_{t\geq 0}$ is a projected gradient flow of $J$ satisfying $\Vert \nu - \nu_{t_0}\Vert_\BL^*< \epsilon$ for some $t_0\geq 0$ and $\nu_{t_0}((J'_{\nu})^{-1}(]-\infty,v]))>0$ then there exists $t_1>t_0$ such that $\Vert \nu- \nu_{t_1}\Vert_\BL^*\geq \epsilon$.
\end{lemma}

\begin{proof}
We first assume that $J'_{\nu}$ takes nonnegative values and let $v\in [2v^\star/3,v^\star/3]$ be a regular value of $g_\nu$, i.e.\ be such that $\Vert \nabla J'_{\nu}\Vert$ does not vanish on the $v$ level-set of $J'_\nu$. Such a $v$ is guaranteed to exist thanks to Morse-Sard's lemma and our assumption that $\phi$ is $d$-times continuously differentiable, which implies that $J'_{\nu}$ is the same. Let $K_v = (J'_{\nu})^{-1}(]-\infty,v])\subset \Theta$ be the corresponding sublevel set. By the regular value theorem, the boundary $\partial K_v$ of $K_v$ is a differentiable orientable compact submanifold of $\Theta$ and is orthogonal to $\nabla J'_{\nu}$. 
By construction, it holds for all $\theta \in K_v$, $J'_{\nu}(\theta) \leq v^\star/3$ and, for some $u>0$, by the regular value property, $\nabla J'_\nu(\theta)\cdot \vec n_{\theta} > u$ for all  $\theta \in \partial K_v$ where $\vec n_\theta$ is the unit normal vector to $\partial K_v$ pointing outwards. Since the map $\nu \mapsto J'_{\nu}$ is locally Lipschitz as a map $(\Mm_+(\Theta), \Vert \cdot \Vert_{\BL}^*) \to (\Cc^1(\Theta),\Vert \cdot\Vert_{\BL})$, there exists $\epsilon>0$ such that if $\nu_t\in \Mm_+(\Theta)$ satisfies $\Vert \nu_t -\nu\Vert_{\BL}^*<\epsilon$, then
\begin{align*}
\forall \theta \in K_v, \quad  J'_{\nu_t}(\theta) \leq v^\star/4 &&\text{and}&& \forall \theta \in \partial K_v,\quad \nabla J'_{\nu_t}(\theta)\cdot \vec n_{\theta} > u/2.
\end{align*}
Now let us consider a projected gradient flow $(\nu_t)_{t\geq 0}$ such that $\Vert \nu_{0} -\nu\Vert_{\BL}^*<\epsilon$ and let $t_1>0$ be the first time such that $\Vert \nu_{t_1}-\nu\Vert_{\BL}^*\geq \epsilon$, which might a priori be infinite. For $t\in {[t_0,t_1[}$, it holds
\[
\frac{\d}{\d t} \nu_t(K_\nu) \geq -4 \alpha \int_{K_v} J'_{\nu_t}\d\nu_t  \geq v^\star \alpha  \nu_t(K_v)
\]
where the first inequality can be seen by using the ``characteristic'' representation of solutions to~\eqref{eq:projectedGF}, see~\cite{maniglia2007probabilistic}. It follows by Gr\"onwall's lemma that $\nu_t(K_v)\geq \exp(\alpha v^\star t)\nu_0(K_v)$ which implies that $t_1$ is finite. Finally, if we had not assumed that $0$ is in the range of $J'_\nu$ in the first place, then we could simply take $K=\Theta$ and conclude by similar arguments.
\end{proof}

\begin{proof}[Proof of Theorem~\ref{th:consistency}]
Let $\nu_\infty\in \Mm(\Theta)$ be the weak limit of $(\nu_t)_t$. It satisfies the stationary point condition $\int \vert J'_{\nu_\infty}\vert^2\d\nu_{\infty}=0$. Then by the optimality conditions in Proposition~\ref{prop:optimalitycondition}, either $\nu_\infty$ is a minimizer of $J$, or $J'_{\nu_\infty}$ is not nonnegative. For the sake of contradiction, assume the latter. Let $\epsilon$ be given by Lemma~\ref{lem:escape} and let $t_0 = \sup \{ t\geq 0\;;\; \Vert \nu_t-\nu_\infty \Vert_{\BL}^*\geq \epsilon\}$ which is finite since we have assumed that $\nu_t$ weakly converges to $\nu_\infty$. But $\nu_{t_0}$ has full support since it can be written as the pushforward of a rescaled version of $\nu_0$ by a diffeomorphism, see~\cite[Eq. (1.3)]{maniglia2007probabilistic} (note that this step is considerably simplified here by the fact that we do not have a potentially non-smooth regularizer, unlike in~\cite{chizat2018global} where topological degree theory comes into play). Then the conclusion of Lemma~\ref{lem:escape} contradicts the definition of $t_0$.
\end{proof}

 \section{Proof of the gradient inequality}\label{app:sharpness}
 In this whole section, we consider without loss of generality $\alpha=\beta=1$ (we explain in Section~\ref{app:conclusionlocal} how to adapt the results to arbitrary $\alpha,\beta$). For simplicity, we only track the dependencies in $\nu$ and $\tau$. Any quantity that is independent of $\nu$ and $\tau$ is treated as a constant and represented by $C,C',C''>0$, and the quantity these symbols refer to can change from line to line.
 
 \subsection{Bound on the transport distance to minimizers} \label{app:bounddistanceA}
Given a measure $\nu\in \Mm_+(\Theta)$, we consider the local centered moments introduced in Definition~\ref{def:localfeatures} and in addition, for $i\in \{1,\dots,m^\star\}$,
\begin{align*}
\delta\theta_i = \bar\theta_i-\theta_i && \tilde b^\theta_i = \bar r_i\delta\theta_i, && s_i \coloneqq \bar r_i(\tr\Sigma_i)^{\frac12}.
\end{align*}
Finally, we will quantify errors with the following quantity
\begin{align}\label{eq:error}
W_\tau(\nu)^2 \coloneqq \bar r_0^2 + \Vert  b\Vert^2 + \Vert s\Vert^2 = \bar r_0^2 + \sum_{i=1}^{m^\star} \left(\vert  b_i^r\vert^2 + \Vert b_i^\theta\Vert^2  + s_i^2 \right)
\end{align}
which also controls the $\widehat W_2$ distance  (introduced in Section~\ref{sec:theoryprelim}) to the minimizer $\nu^\star$ of $J$, as shown in the next proposition.
\begin{lemma}\label{lem:distanceinequality}
It holds $\widehat W_2(\nu,\nu^\star)\leq W_\tau(\nu)(1+O(\tau^2)+O(W_\tau(\nu)^2))$.
\end{lemma}
\begin{proof}
Note that for $W_\tau(\nu)$ small enough, it holds $\nu(\Theta_i)>0$ for $i\in \{1,\dots,m^\star\}$. Let $\mu \in \Pp_2(\Omega)$ be such that $\h\mu = \nu$ and consider the transport map $T:\Omega\to\Omega$ defined as
 \[
 T(r,\theta) =
 \begin{cases}
(r\frac{r_i}{\bar r_i}, \theta_i) &\text{if $\theta \in \Theta_i$ and $\bar r_i>0$, $i\in \{1,\dots,m^\star\}$},\\
(0,\theta) &\text{otherwise}.
 \end{cases}
 \]
By construction, it holds $\h(T_\#\mu) =  \nu^\star$. Let us estimate the transport cost associated to this map 
\[
\mathcal{T} = \int_\Omega \dist((r,\theta),T(r,\theta))^2\d\mu(r,\theta).
\]
The geodesic distance associated to the cone metric is
\begin{align*}
\dist((r_1,\theta_1),(r_2,\theta_2))^2 &= r_1^2 + r_2^2 - 2r_1r_2\cos(\min\{\dist(\theta_1,\theta_2),\pi\})\\
&= (r_1-r_2)^2 +r_1r_2\dist(\theta_1,\theta_2)^2 +O(r_1r_2\dist(\theta_1,\theta_2)^4)
\end{align*}
Now, if we only consider points $\theta\in \Theta_i$ with $\tilde \theta$ their coordinates in a normal frame centered at $\theta_i$ (note that in all other proofs, we do not need to distinguish between $\theta$ and $\tilde \theta$), we have the approximation
\begin{align*}
\dist((r_1,\theta_1),(r_2,\theta_2))^2 = (r_1-r_2)^2 +r_1r_2\Vert \tilde \theta_1 - \tilde \theta_2\Vert^2(1+O(\tau^2)).
\end{align*}
Let us decompose $T(r,\theta)$ as $(rT^r(\theta),T^\theta(\theta))$ and estimate the two contributions forming $\Tt$ separately. On the one hand, we have
\begin{align*}
\int (rT^r(\theta)-r)^2\d\mu(r,\theta) &= \bar r_0^2 + \sum_{i=1}^{m^\star} (\bar r_i-r_i)^2 = \bar r_0^2 +\Vert b^r\Vert^2(1+O(W_\tau(\nu)^2)).
\end{align*}
On the other hand, we have
\begin{align*}
\int r^2T^r(\theta)\Vert T^\theta(\theta)-\theta\Vert^2\d\mu(r,\theta) &=\sum_{i=1}^{m^\star} \bar r_ir_i (\tr\Sigma_i + \Vert \delta \theta_i\Vert^2)=(\Vert s\Vert^2 + \Vert b^\theta\Vert^2)(1+O(W_\tau(\nu)^2)).
\end{align*}
As a consequence, we have $\Tt = W_\tau(\nu)(1 + O(W_\tau(\nu)^2)+ O(\tau^2))$. Remark that this estimate does not depend on the chosen lifting $\mu$ satisfying $\h\mu=\nu$. We then conclude by using the characterization in~\cite[
Thm. 7.20]{liero2018optimal} for the distance $\widehat W_2$:
\[
\widehat W_2(\nu_1,\nu_2) = \min \left\{W_2(\mu_1,\mu_2) \;;\; (\h\mu_1,\h\mu_2) = (\nu_1,\nu_2)\right\}.
\]
Thus $\widehat W_2(\nu,\nu^\star)^2 \leq W_2(\mu,T_\#(\mu))^2\leq \Tt$, and the result follows.
\end{proof}

 \subsection{Local expansion lemma} 
 \begin{lemma}[Expansion around $\nu^\star$]\label{lem:expansion}
 Let $\psi$ be any (vector or real-valued) smooth function on $\Theta$ and $\nu\in \Mm_+(\Theta)$. If $\tau>0$ is an admissible radius, then the following first and second-order expansions hold
 \begin{align*}
 \int \psi \d(\nu - \nu^\star) 
 &= \sum_{i=1}^m r_i \bar \nabla \psi(\theta_i)^\intercal b_i 
 + \int_{\Theta_0} \psi \d\nu 
 +\sum_{i=1}^m \int_{\Theta_i} M_{2,\psi}(\theta_i,\theta)\d\nu(\theta)\\
 &=\sum_{i=1}^m r_i\bar \nabla \psi(\theta_i)^\intercal b_i 
 +\frac12  \sum_{i=1}^m \bar r_i^2  \left( \tr\left(\nabla^2\psi(\theta_i) \Sigma_i \right) 
 + \delta \theta_i^\intercal \nabla^2\psi(\theta_i) \delta \theta_i \right)\\
 &\qquad + \int_{\Theta_0}  \psi \d\nu +\sum_{i=1}^m \int_{\Theta_i} M_{3,\psi}(\theta_i,\theta)\d\nu(\theta)
 \end{align*}
 where $M_{k,\psi}(\theta_i,\theta)$ is the remainder in the $k-1$-th order Taylor expansion of $\psi$ around $\theta_i$ in local coordinates (and we recall that $\bar \nabla \psi := (2\psi,\nabla \psi)$).
 \end{lemma}
 \begin{proof}
By a Taylor expansion of $\psi$ around $\theta_i$ for $i\in \{1,\dots,m^\star\}$, it holds
\begin{align*}
 \int_{\Theta_i} \psi \d\nu= \int_{\Theta_i} \left( \psi(\theta_i) + \nabla\psi(\theta_i)^\intercal (\theta-\theta_i) + (\theta-\theta_i) ^\intercal \nabla^2\psi(\theta_i)(\theta-\theta_i) + M_{3,\psi}(\theta_i,\theta)\right)\d\nu(\theta)
\end{align*}
and substracting $\int_{\Theta_i}\psi\d\nu^\star = r_i^2\phi(\theta_i)$ yields
\begin{multline*}
 \int_{\Theta_i} \psi \d(\nu-\nu^\star)= (\bar r_i^2 -r_i^2)\psi(\theta_i) +  \bar r_i^2 \nabla \psi(\theta_i)^\intercal \delta \theta_i \\ 
 +\frac12  \sum_{i=1}^m \bar r_i^2  \left( \tr\left(\nabla^2\psi(\theta_i) \Sigma_i \right) 
 + \delta \theta_i^\intercal \nabla^2\psi(\theta_i) \delta \theta_i \right)  +\sum_{i=1}^m \int_{\Theta_i} M_{3,\psi}(\theta_i,\theta)\d\nu(\theta)
 \end{multline*}
where  we have used a bias-variance decomposition for the quadratic term. The result follows by summing the integrals over each $\Theta_i$ and using the expression of $b$.
 \end{proof}

  \subsection{Bound on the distance to minimizers} 
 In the next lemma, we globally bound the quantity $W_\tau(\nu)$ from Eq.~\eqref{eq:error} in terms of the function values. It involves the quantity $v^\star>0$ which is such that for any local minimum $\theta$ of $J'_{\nu^\star}$, either $\theta = \theta_i$ for some $i\in \{1,\dots,m^\star\}$ or $J'_{\nu^\star}(\theta)\geq v^*$ (which is non-zero under {\sf (A5)}). We also recall that $\tilde b^\theta_i = \bar r_i \delta\theta_i$, as defined in Section~\ref{app:bounddistanceA}.
 
\begin{lemma}[Global distance bound]\label{lem:distancebound}
Under {\sf (A1-5)}, let $\tau_{\adm}$ be an admissible radius $\tau$ as in Definition~\ref{def:localfeatures}, fix some $J_{\max}>0$ and let
\[
\tau_0 = \min\left\{\tau_{\adm}, 2\sqrt{\frac{v^\star}{\sigma_{\min}(H)}}, \frac{3\sigma_{\min}(H)}{2 \Lip (\nabla^2 J'_{\nu^\star})} \right\}.
\]
Then there exists $C, C'>0$ such that for all $\tau\leq \tau_0$ and  $\nu \in \Mm_+(\Theta)$ such that $J(\nu)\leq J_{\max}$, it holds 
\begin{align*}
W_\tau(\nu) \leq \frac{C}{\tau^2}(J(\nu)-J^\star)^\frac12 &&\text{and}&&\Vert \tilde b^\theta\Vert^2+\Vert s\Vert^2\leq C'(J(\nu)-J^\star). 
\end{align*}
\end{lemma}

\begin{proof}
Let us write $f_\nu \coloneqq \int \phi\d\nu$ and $f^\star = \int \phi \d\nu^\star$. By strong convexity of $R$ at $f^\star$, and optimality of $\mu^\star$, there exists $C>0$ such that for all $\nu\in \Mm_+(\Theta)$ it holds
\begin{equation}\label{eq:strongconvexityinequality}
J(\nu)-J^\star \geq \int_{\Theta} J'_{\nu^\star}\d\nu+ C\min\{ \Vert f_\nu-f^\star\Vert^2,\Vert f_\nu-f^\star\Vert\} .
\end{equation}
To prove the first claim, we thus have to bound $W_\tau(\nu)$ using the terms in the right-hand side of~\eqref{eq:strongconvexityinequality}. 

\paragraph{Step 1.} By a Taylor expansion, one has for $\theta \in \Theta_i$ for $i\in \{1,\dots,m^\star\}$,
\[
\vert J'_{\nu^\star}(\theta) - \tfrac12 (\theta -\theta_i)^\intercal H_i  (\theta-\theta_i) \vert \leq \tfrac16 \Lip( \nabla^2 J'_{\nu^\star}) \Vert \theta-\theta_i\Vert^3.
\]
Thus, if $\Vert \theta - \theta_i\Vert \leq 3\sigma_{\min}(H)/(2 \Lip( \nabla^2 J'_{\nu^\star}))$, then $J'_{\nu^\star}(\theta) \geq \frac14 (\theta -\theta_i)^\intercal H_i  (\theta-\theta_i)$ for $\theta\in \Theta_i$. Decomposing the integral of this quadratic term into bias and variance, we get
\begin{align*}\label{eq:biasvariance}
\int_{\Theta_i} (\theta -\theta_i)^\intercal H_i (\theta-\theta_i)\d\nu(\theta) 
&= \bar r_i^2 \left(\delta \theta_i^\intercal  H_i \delta \theta_i + \tr(\Sigma_i H_i )\right)
\end{align*}
and we deduce a first bound by summing the terms for $i\in \{1,\dots,m^\star\}$,
\[
\int_{\Theta\setminus \Theta_0} J'_{\nu^\star} \d\nu \geq \frac{\sigma_{\min}(H)}4(\Vert \tilde b^\theta\Vert^2 + \Vert s\Vert^2).
\]

\paragraph{Step 2.} In order to lower bound the integral over $\Theta_0$, we first derive a lower bound for $J'_{\nu^\star}$ on $\Theta_0$. This is a continuously differentiable and nonnegative function on a closed domain $\Theta_0$ so its minimum is attained either at a local minima in the interior of $\Theta_0$ or on its boundary. Using the quadratic lower bound from the previous paragraph, it follows that for $\theta\in \Theta_0$,
\[
J'_{\nu^\star}(\theta)\geq \min \{ v^\star,  \tau^2 \sigma_{\min}(H)/4\}.
\]
Thus, if we also assume that $\tau\leq 2\sqrt{v^\star/\sigma_{\min}(H)}$ then $J'_{\nu^\star}(\theta)\geq \tau^2 \sigma_{\min}(H)/4$ for $\theta\in \Theta_0$ and it follows that
\[
\int_{\Theta_0} J'_{\nu^\star} \d\nu \geq \frac{\sigma_{\min}(H)}{4} \tau^2\bar r_0^2 .
\]
Using inequality~\eqref{eq:strongconvexityinequality}  we have shown so far that
\begin{equation}\label{eq:localcontrol}
\tilde W_\tau(\nu)^2\coloneqq  \bar r_0^2 +\Vert \tilde b^\theta\Vert^2 +\Vert s\Vert^2 \leq \frac{C}{\tau^2} (J(\mu)-J^\star).
\end{equation}
Notice that $\tilde W_\tau(\nu)$ is similar to $W_\tau(\nu)$ but it does not contain the terms controlling the deviations of mass $\vert \bar r_i-r_i\vert$. These quantities can be controlled by using the coercivity of $R$, i.e. the last term in~\eqref{eq:strongconvexityinequality}, as we do now.

\paragraph{Step 3.}
Using the first order expansion of Lemma~\ref{lem:expansion} then squaring gives
\begin{align*}
\left\vert \Vert f_\nu-f^\star\Vert^2 - \frac12 b^\intercal K b \right\vert
& \leq C \left( \Vert b\Vert \tilde W_\tau(\nu)^2 +  \tilde W_\tau(\nu)^4 \right).
\end{align*}
Since we have assumed that $K$ is positive definite, it follows
\[
\Vert f_\nu-f^\star\Vert^2 \geq C \Vert b\Vert^2 - C \tilde W_\tau(\nu)^2 \Vert b\Vert - C\tilde W_\tau(\nu)^4
\]
and thus, after rearranging the terms
\[
(\Vert b\Vert - C\tilde W_\tau(\nu)^2)^2 \leq  C \Vert f_\nu-f^\star\Vert^2 + C\tilde W_\tau(\nu)^4.
\]
It follows that $\Vert b\Vert \leq  C\Vert f_\nu-f^\star\Vert + C\tilde W_\tau(\nu)^2$. Also, by inequality~\eqref{eq:strongconvexityinequality}, if $J(\nu)\leq J_{\max}$, then $\Vert f_\nu -f^\star\Vert^2\leq C(J(\nu)-J^\star)$.  Moreover, by inequality~\eqref{eq:localcontrol}, we get
\[
\Vert b\Vert \leq \frac{C}{\tau^2}(J(\nu)-J^\star) +C (J(\nu)-J^\star)^\frac12 \leq \frac{C}{\tau^2} (J(\nu)-J^\star)^\frac12.
\]
We finally combine with the bound on $\tilde W_\tau(\nu)$ to conclude since $W_\tau(\nu)^2\leq \tilde W_\tau(\nu)^2+\Vert b\Vert^2$
\end{proof} 

\subsection{Proof of the distance inequality (Proposition~\ref{prop:distanceinequality})}\label{app:distanceinequality}
By Lemma~\ref{lem:distanceinequality}, it holds 
\[
\widehat W_2(\nu,\nu^\star)\leq W_\tau(\nu)(1+O(\tau^2)+O(W_\tau(\nu)^2)).
\]
Moreover, by Lemma~\ref{lem:distancebound}, there exists $\tau_0>0$ and $C>0$ such that 
\[
W_\tau(\nu)\leq \frac{C}{\tau_0^2}(J(\nu)-J^\star)^\frac12.
\]
Combining these two lemmas, it follows that for some $C'>0$, we have 
\[
\widehat W_2(\nu,\nu^\star)^2 \leq C'(J(\nu)-J^\star)^\frac12.
\]
 This also implies a control on the Bounded-Lipschitz distance since it holds $(\Vert \nu - \nu^\star\Vert_{\BL}^*)^2\leq (2+\pi^2/2)(\nu(\Theta)+\nu^\star(\Theta))\widehat W_2(\nu,\nu^\star)^2$, see~\cite[Prop. 7.18]{liero2018optimal}.

 \subsection{Local estimate of the objective}
 We now prove a local expansion formula for $J$.
 \begin{proposition}[Local expansion]\label{prop:localexpansion}
 It holds
 \[
 J(\nu)-J^\star = \frac12 b^\intercal K b + \frac12 \sum_{i=1}^m \bar r_i^2 \left( \tr(\Sigma_i H_i) +\delta \theta^\intercal H_i \delta \theta_i \right) + \int_{\Theta_0} J'_{\nu^\star} \d\nu + \mathop{\mathrm{err}}(\tau,\nu)
 \]
 where $\mathop{\mathrm{err}}(\tau,\nu) =O( \tau (\Vert \tilde b^\theta \Vert^2 + \Vert s\Vert^2) + W_\tau(\nu)^3)$. In particular, if $\tau$ is fixed small enough,
 \[
 J(\nu)-J^\star \leq  \sigma_{\max}(K) \Vert b\Vert^2 + \sigma_{\max}(H)(\Vert b^\theta\Vert^2 +\Vert s\Vert^2) +\Vert J'_{\nu^\star}\Vert_\infty \bar r_0^2 +O(W_\tau(\nu)^3).
 \]
 \end{proposition}
 \begin{proof} Let us write $f_\nu \coloneqq \int \phi\d\nu$ and $f^\star = \int \phi \d\nu^\star$. By a second order Taylor expansion of $R$ around $f^\star$, we have
 \[
 J(\nu) - J^\star = \int_{\Theta} J'_{\nu^\star} \d\nu+ \frac12 \Vert f_\nu-f^\star\Vert^2_\star + O(\Vert f_\nu-f^\star\Vert^3).
 \]
 Using the first order expansion of Lemma~\ref{lem:expansion} for $\phi$, we get 
 $
 \Vert f_\nu-f^\star\Vert^2_\star 
 = b^\intercal Kb + O(W_{\tau}(\nu)^3)
 $. Also, using the second order expansion of Lemma~\ref{lem:expansion} for $J'_{\nu^\star}$ and using the fact that $J'_{\nu^\star}$ and its gradient vanish for all $\theta_i$, we get 
 \[
 \int_{\Theta} J'_{\nu^\star} \d\nu = \frac12 \sum_{i=1}^m \bar r_i^2 \left( \tr(\Sigma_i H_i) +\delta \theta^\intercal H_i \delta \theta_i \right) + \int_{\Theta_0}J'_{\nu^\star} \d\nu  + O(\tau (\Vert s\Vert^2+\Vert  b^\theta \Vert^2))
 \]
and the expansion follows. Notice also that in the expression of $J(\nu)$, $\bar r_i$ and $r_i$ are interchangeable up to introducing higher order error, since $\vert r_i - \bar r_i\vert = O(\vert b^r_i\vert)$ (and also $\Vert \tilde b^\theta\Vert = \Vert b^\theta\Vert(1+O(W_\tau(\nu)))$).
 \end{proof}
 
 
 \subsection{Local estimate of the gradient norm}
 
\begin{proposition}[Gradient estimate]\label{prop:gradientestimate}
For $\nu \in \Pp_2(\Omega)$, it holds 
\[
\Vert g_\nu\Vert^2_{L^2(\nu)} =  b^\intercal (K + H)^2b+ \sum_{i=1}^m \bar r_i^2 \tr (\Sigma_i H_i^2) +\Vert g_\nu\Vert^2_{L^2(\nu\vert_{\Theta_0})}+\mathop{\mathrm{err}}(\tau,\nu)
\]
where $\mathop{\mathrm{err}}(\tau,\nu)\lesssim \tau(\Vert \tilde b^\theta\Vert^2 +\Vert s\Vert^2) +W_\tau(\nu)^3$. In particular, if $\tau$ is fixed small enough
\[
\Vert g_\nu\Vert^2_{L^2(\nu)}  \geq \frac12 (\sigma_{\min}(K)+ \sigma_{\min}(H))^2\Vert b\Vert^2 + \frac12 \sigma_{\min}(H)^2 \Vert s\Vert^2 +  \frac14 \bar r_0^2 \sigma_{\min}(H)^2 \tau^4 +O(W_\tau(\nu)^3).
\]
\end{proposition}
\begin{proof}
For this proof, we write $f_\nu - f^\star = \delta f_0 +\delta f_b + \delta f_{\mathrm{err}}$ where
\begin{align*}
 \delta f_0 &\coloneqq  \int_{\Theta_0}\phi(\theta)\d\nu(\theta),&
 \delta f_b & \coloneqq  \sum_{i=1}^m r_i \bar\nabla \phi(\theta_i) b_{i},&
\delta f_{\mathrm{err}} &\coloneqq \sum_{i=1}^{m} \int_{\Theta_i} M_{\phi,2}(\theta_i,\theta)\d\nu(\theta).
\end{align*}
where the decomposition follows from Lemma~\ref{lem:expansion}.
The expression for the norm of the gradient is as follows:
\[
\Vert g_\nu\Vert^2_{L^2(\nu)} = \int_{\Theta}\Vert\bar \nabla J'_\nu(\theta) \Vert^2\d\nu(\theta)
\]
where $\bar \nabla J = (2J,\nabla J)$. We start with the following decomposition for $\theta \in \Theta_i$ (recall that $J'_\nu(\theta) =\langle \phi(\theta),\nabla R(\int \phi \d\nu)\rangle +\lambda$):
\begin{multline*}
J'_\nu(\theta) = \lambda + \left\langle \phi(\theta_i) + (\theta-\theta_i)^\intercal \nabla \phi(\theta_i) + \frac12  (\theta-\theta_i)^\intercal\nabla^2\phi(\theta_i)(\theta-\theta_i)+  M_{\phi,3}(\theta_i,\theta),\nabla R(f^\star)\right\rangle \\
+\langle \phi(\theta_i)+ (\theta-\theta_i)^\intercal \nabla \phi(\theta_i) + M_{\phi,2}(\theta_i,\theta),f_\mu-f^\star\rangle_{\star}+ \langle \phi(\theta),M_{\nabla R,2}(f^\star,f)\rangle
\end{multline*}
Here we use the notation $\langle\cdot,\cdot\rangle_\star$ to denote the quadratic form associated to $\nabla^2 R(f^\star)$.
Thanks to the optimality conditions  $\bar \nabla J'_{\nu^\star}(\theta_i)=0$ for $i\in \{1,\dots,m\}$, we get
\begin{align*}
\bar \nabla_j J'_\nu(\theta) &=  [H_i(\theta-\theta_i)]_j + \langle \bar \nabla_j \phi(\theta_i),\delta f_0+\delta f_b \rangle_\star + [N(\theta_i,\theta)]_j\\
 &=  [H_i(\theta-\bar \theta_i)]_j+ \left( [H_i(\bar \theta_i-\theta_i)]_j+ \langle \bar \nabla_j \phi(\theta_i),\delta f_b \rangle_\star\right) + \langle \bar \nabla_j  \phi(\theta_i),\delta f_0\rangle_\star + [N(\theta_i,\theta)]_j
\end{align*}
where $N$ collects the higher order terms and is defined as
\begin{multline*}
 [N(\theta_i,\theta)]_j =  \langle \bar \nabla_j M_{\phi,3}(\theta_i,\theta),\nabla R(f^\star)\rangle 
 +\langle \bar \nabla_j M_{\phi,2}(\theta_i,\theta),f-f^\star\rangle_{\star} + \langle \bar \nabla_j \phi(\theta),M_{\nabla R,2}(f^\star,f)\rangle
\end{multline*}
where $\Vert \bar \nabla_j M_{\phi,3}(\theta_i,\theta)\Vert = O(\Vert \theta-\theta_i\Vert^2)$ if $j>0$ and $O(\Vert \theta -\theta_i\Vert^3)$ if $j=0$.
Expanding the square gives the following ten terms:
\begin{align}
\int_{\Omega\setminus \Omega_0} \Vert\bar \nabla J'_\nu \Vert^2\d\nu(\theta)
&= \sum_{i=1}^{m} \int_{\Theta_i} \sum_{j=0}^{d} [H_i(\theta-\bar \theta_i)]_j^2 \d\nu \tag{I} \\
&+ \sum_{i=1}^{m}\int_{\Theta_i} \sum_{j=0}^d \left( \langle\bar \nabla_j\phi(\theta_i),\delta f_b \rangle_\star + [H_i(\bar \theta_i-\theta_i)]_j \right)^2 \d\nu(\theta) \tag{II} \\
&+ \sum_{i=1}^{m} \int_{\Theta_i} \sum_{j=0}^d \langle \bar \nabla_j\phi(\theta_i), \delta f_0 \rangle_\star^2\d\nu(\theta)  \tag{III} \\
&+ \sum_{i=1}^{m}\int_{\Theta_i} \!\! 2 \sum_{j=0}^d  [H_i(\theta-\bar \theta_i)]_j\cdot ( \langle\bar \nabla_j \phi(\theta_i),\delta f_b\rangle_\star + [H_i(\bar \theta_i-\theta_i)]_j) \d\nu(\theta) \tag{IV} \\
&+ \sum_{i=1}^{m}\int_{\Theta_i} \!\! 2 \sum_{j=0}^d [H_i(\theta-\bar \theta_i)]_j\cdot \langle \bar \nabla_j\phi(\theta_i), \delta f_0\rangle_\star\d\nu(\theta) \tag{V}\\
&+ \sum_{i=1}^{m}\int_{\Theta_i} \!\! 2 \sum_{j=0}^d (\langle\bar \nabla_j \phi(\theta_i),\delta f_b\rangle_\star + [H_i(\bar \theta_i-\theta_i)]_j)\cdot \langle \bar\nabla_j \phi(\theta_i), \delta f_0\rangle_\star\d\nu(\theta)\!\! \tag{VI}\\
&+ \sum_{i=1}^{m}\int_{\Theta_i} \!\! 2 \sum_{j=0}^d [N(\theta_i,\theta)]_j\cdot  [H_i(\theta-\bar \theta_i)]_j\d\nu(\theta) \tag{VII}\\
&+ \sum_{i=1}^{m}\int_{\Theta_i} \!\! 2 \sum_{j=0}^d [N(\theta_i,\theta)]_j\cdot  (\langle\bar \nabla_j \phi(\theta_i),\delta f_b\rangle_\star + [H_i(\bar \theta_i-\theta_i)]_j)\d\nu(\theta) \tag{VIII}\\
&+\sum_{i=1}^{m}\int_{\Theta_i} \!\! 2 \sum_{j=0}^d [N(\theta_i,\theta)]_j\cdot \langle \bar \nabla_j \phi(\theta_i), \delta f_0 \rangle_\star\d\nu(\theta) \tag{IX}\\
&+ \sum_{i=1}^{m} \int_{\Theta_i} \sum_{j=0}^d [N(\theta_i,\theta)]_j^2 \d\nu(\theta) \tag{X}
\end{align}
Terms (I) to (II) are the main terms in the expansion, while the other terms are higher order. The term (I) is a local curvature term and can be expressed as $\mathrm{(I)} = \sum_{i=1}^m \bar r_i^2 \tr \Sigma_i H_i^2$. 
The term (II) is a global interaction term that writes
\begin{align*}
\mathrm{(II)} &=
\sum_{i=1}^m \bar r_i^2 \sum_{j=0}^d \vert \langle\bar \nabla_j\phi(\theta_i),\delta f_b\rangle_\star + [H_i(\bar \theta_i-\theta_i)]_j \vert^2 \\
&= \sum_{i=1}^m\sum_{j=0}^d \left\vert \sum_{i'=1}^m \sum_{j'=0}^d (\langle \bar r_i \bar \nabla_j\phi(\theta_i),r_{i'}\bar \nabla_{j'}\phi(\theta_{i'})\rangle_\star + \bar H_{(i,j),(i,j')})(b_{i',j'}) \right\vert^2\\
&= \Vert (\bar K + H)(b)\Vert^2
\end{align*}
where the entries of $\bar K$ and $\bar H$ differ from those of $K$ and $H$ by a factor $\bar r_i/r_i$. More precisely,
\[
[\bar K-K]_{(i,j),(i',j')} = (\bar r_i/r_i-1)K_{(i,j),(i',j')}
\]
and similarly for $\bar H- H$. Since $\vert \bar r_i/r_i-1\vert = O(\vert b^r_i\vert)$ we have $\sigma_{\max}(\bar K-K )=O(W_\tau(\nu))$. It follows, by expanding the square, that 
\[
\Vert (\bar K + \bar H)(b)\Vert^2 = \Vert (K + H)(b)\Vert^2 + O(W_\tau(\nu)^3).
\]

The remaining terms are error terms, that we estimate directly in terms of $W_\tau(\nu)$ and $\tau$. We use in particular the fact that by H\"older's inequality, $\int_{\Theta_i} \Vert \theta - \bar \theta_i\Vert\d\nu(\theta) = O(\bar r_i^2 \tr \Sigma_i^{\frac12})$. One has 
\begin{itemize}
\item $\mathrm{(III)} =O\left(\sum_{i=1}^m \bar r_i^2 \bar r_0^4 \right)= O(W_{\tau}^4(\nu))$;
\item $\mathrm{(IV)} = \mathrm{(V)} = 0$ because the integral of the terms $H_i(\theta-\bar \theta_i)$ vanishes;
\item $\mathrm{(VI)} = O\left((\sum_{i=1}^m \bar r_i^2 (\Vert b\Vert +\Vert \delta \theta_i\Vert)\cdot \bar r_0^2\right)=O(W_{\tau}^3(\nu))$;
\item $\mathrm{(VII)} = O(\tau (\Vert \tilde b^\theta\Vert^2 + \Vert s\Vert^2)) + O(W_{\tau}(\nu)^3)$;
\item $\mathrm{(VIII)} = O(W_{\tau}^3(\nu))$;
\item $\mathrm{(IX)} = O(W_{\tau}^4(\nu))$;
\item $\mathrm{(X)} = O(\tau^2(\Vert \tilde b^\theta\Vert^2+\Vert s\Vert^2))+O(W_{\tau}(\nu)^4)$.
\end{itemize}
It follows that overall, the error term is in $O(\tau (\Vert \tilde b^\theta\Vert^2+\Vert s\Vert^2)+W_{\tau}(\nu)^3)$.
There remains to lower bound the norm of the gradient over $\Theta_0$, which can be done as follows. As seen in the proof of Lemma~\ref{lem:distancebound}, if $\tau$ is small enough then $J'_{\nu^\star}(\theta)\geq \tau^2\sigma_{\min}(H)/4$ for $\theta \in \Theta_0$. Considering only the first component of the gradient, it holds
\[
\int_{\Theta_0} \Vert\bar \nabla J'_\nu(\theta) \Vert^2\d\nu(\theta) \geq \int_{\Theta_0} 4 \vert J'_\nu(\theta)\vert^2 \d\nu(\theta).
\]
Using the expansion $J'_\nu(\theta)=J'_{\nu^\star}(\theta)+\langle \phi(\theta),M_{\nabla R,1}(f^\star,f_\nu)\rangle$, we get
\[
\int_{\Theta_0} \Vert\bar \nabla J'_\nu(\theta) \Vert^2\d\nu(\theta) \geq  C \bar r_0^2 \tau^4 + O(W_\tau(\nu)^3).
\]
The result follows by collecting all the estimates above.
\end{proof}

\subsection{Proof of the sharpness inequality (Theorem~\ref{th:gradientinequality})}\label{app:conclusionlocal}
By Proposition~\ref{prop:localexpansion} we have that for $\tau>0$ small enough
\[
 J(\nu)-J^\star \leq CW_\tau(\nu)^2 + O(W_\tau(\nu)^3)
\]
where $C = \sigma_{\max}(K+H) +\Vert J'_{\nu^\star}\Vert_\infty$.

Similarly, by Proposition~\ref{prop:gradientestimate}, for $\tau$ small enough, it holds
\[
\Vert g_\nu\Vert^2_{L^2(\nu)}  \geq C'W_\tau(\nu)^2 +O(W_\tau(\nu)^3)
\]
where $C'= \frac18\sigma_{\min}(H)^2\tau^4$. Now fix $\tau>0$ satisfying the hypothesis of Lemma~\ref{lem:distancebound} and the two previous inequalities. By Lemma~\ref{lem:distancebound},  $W_\tau(\nu) = O((J(\nu)-J^\star)^\frac12)$. We deduce that there exists $J_0>J^\star$ and $\kappa_0>0$ , such that whenever  $\nu\in \Mm_+(\Theta)$ satisfies $J(\nu)<J_0$, one has
\[
\Vert g_\nu\Vert^2_{L^2(\nu)}\geq \kappa_0(J(\nu)-J^\star).
\]
Finally, notice that if different metric factors $(\alpha,\beta)\neq (1,1)$ are introduced, one can always lower bound the new gradient squared norm as
\[
\int_\Theta \left( 4 \alpha \vert J'_\nu(\theta)\vert^2 +\beta \Vert \nabla J'_\nu(\theta)\Vert^2_\theta \right)\d\nu(\theta)
\geq \min\{\alpha,\beta\} \int_\Theta \left( 4 \vert J'_\nu(\theta)\vert^2 + \Vert \nabla J'_\nu(\theta)\Vert^2_\theta \right)
 \]
 which proves the statement for any $(\alpha,\beta)$. Note however that if one wants to make a more quantitative bound, then there are values $(\alpha_0,\beta_0)$ that would lead to a better conditioning and potentially higher values for $J_0$. In this case, the factor appearing in the sharpness inequality should rather be $\min\{\alpha/\alpha_0,\beta/\beta_0\}$.

\section{Estimation of the mirror rate function}\label{app:mirror}
We provide an upper bound for the  \emph{mirror rate} function $\Qq$ in the situation that is of interest to us, with $\nu^\star$ sparse. Note that this approach could be generalized to arbitrary $\nu^\star$.
\begin{lemma}\label{lem:qdensity}
Under {\sf (A1)}, there exists $C_{\Theta}> 0$ that only depends on the curvature of $\Theta$, such that for all $\nu^\star ,\nu_0\in \Mm_+(\Theta)$ where $\nu^\star=\sum_{i=1}^{m^\star} r_i^2\delta_{\theta_i}$ and $\nu_0=\rho \vol$ where $\log \rho$ is $L$-Lipschitz, then 
\[
\Qq_{\nu^\star,\nu_0}(\tau)\leq \frac1{\tau} \left( \nu^\star(\Theta)\cdot d\cdot \big(C_\Theta +\log(\tau)+L/\tau\big) +\nu_0(\Theta)- \nu^\star(\Theta)+ \sum_{i=1}^{m^\star} r_i^2 \log\left(\frac{r_i^2}{\rho(\theta_i)}\right) \right).
\]
Moreover, for any other $\hat \nu_0\in \Mm_+(\Theta)$, it holds
$
\Qq_{\nu^\star,\hat \nu_0}(\tau)\leq \Qq_{\nu^\star,\nu_0}(\tau) + \nu^\star(\Theta)\cdot W_\infty(\nu_0,\hat \nu_0).
$
\end{lemma}

In the context of Lemma~\ref{lem:qdensity}, we introduce the quantity,
\[
\bar \Hh(\nu^\star,\rho) \coloneqq \sum_{i=1}^{m^\star} r_i^2 \log\left(\frac{r_i^2}{\rho(\theta_i)}\right) -  \nu^\star(\Theta) + \nu_0(\Theta).
\]
which measures how much $\rho$ is a good prior for the (a priori unknown) minimizer $\nu^\star$. With this quantity, the conclusion of Lemma~\ref{lem:qdensity} reads, for $\tau\geq L$,
\[
\Qq_{\nu^\star,\nu_0}(\tau)\leq \frac{\bar \Hh(\nu^\star,\rho) + \nu^\star(\Theta)\cdot d\cdot (\log(\tau)+C_\Theta)}{\tau}.
\]

\begin{proof}
Let us build $\nu_\epsilon$ in such a way that the quantity defining $\Qq_{\nu^\star,\nu_0}(\tau)$ in Eq.~\eqref{eq:quantization} is small. For this, consider a radius $\epsilon>0$ and consider the measure $\nu_\epsilon$ defined as the normalized volume measure on each geodesic ball of radius $\tau$ around each $\theta_i$, with mass $r_i^2$ on this ball, and vanishing everywhere else. Using the transport map that maps these balls to their centers $\theta_i$, we get if $\Theta$ is flat,
\[
\Vert \nu_\epsilon -\nu^\star\Vert_\BL^* \leq W_1(\nu_\epsilon,\nu^\star)  \leq \sum_{i=1}^{m^\star} \frac{r_i^2}{V^{(d)}(\epsilon)} \int_0^\epsilon s \frac{d}{ds}V^{(d)}(s)\d s
\]
where $V^{(d)}(\epsilon)$ is the volume of a ball of radius $\epsilon$ in $\RR^d$, that scales as $\epsilon^d$. Using an integration by parts, it follows
\[
\frac{1}{V^{(d)}(\epsilon)} \int_0^\epsilon s \frac{d}{ds}V^{(d)}(s)\d s = \epsilon - \int_0^\epsilon \frac{V^{(d)}(s)}{V^{(d)}(\epsilon)}\d s = \epsilon -\int_0^\epsilon\left(\frac{s}{\epsilon}\right)^d\d s = \frac{\epsilon d}{d+1},
\]
thus $W_1(\nu_\epsilon ,\nu^\star)  \leq \nu^\star(\Theta)\epsilon$. In the general case where $\Theta$ is a potentially curved manifold, this upper bound also depends on the curvature of $\Theta$ around each $\theta_i$, a dependency that we hide in the multiplicative constant so $W_1(\nu_\epsilon ,\nu^\star)  \leq C\nu^\star(\Theta)\epsilon$. Let us now control the entropy term. Writing $\rho_\epsilon =\d\nu_\epsilon/\d{\vol}$ and $\Theta_i$ for the geodesic ball of radius $\epsilon$ around $\theta_i$, it holds
\begin{align*}
\Hh(\nu_\epsilon,\nu_0) & = \nu_0(\Theta) -\nu^\star(\Theta) + \sum_{i=1}^{m^\star} \int_{\Theta_i} \rho_\epsilon(\theta)\log (\rho_\epsilon(\theta)/\rho(\theta))\d{\vol}(\theta).
\end{align*}
The integral term can be estimated as follows,
\begin{align*}
\int_{\Theta_i} \rho_\epsilon(\theta)  \log (\rho_\epsilon(\theta)/\rho(\theta))\d{\vol}(\theta) &= \int_{\Theta_i} \frac{r_i^2}{V^{(d)}(\epsilon)} \left(\log (r_i^2)-\log V^{(d)}(\epsilon) -\log(\rho(\theta))\right)\d{\vol}(\theta)\\
& \leq r_i^2\left(\log(r_i^2)-\log V^{(d)}(\epsilon) -\log \rho(\theta_i) +\Lip(\log \rho)\cdot \epsilon \right).
\end{align*}
Recalling that $-\log V^{(d)}(\epsilon) \leq -d\log(\epsilon)+C$ for some $C$ that only depends on the curvature of $\Theta$, we get that the right-hand side of~\eqref{eq:quantization} is bounded by
\begin{align*}
 \frac1{\tau} \left(C\nu^\star(\Theta)+\nu_0(\Theta) -\nu^\star(\Theta) d\log(\epsilon) + \sum_{i=1}^{m^\star} r_i^2 \log\left(\frac{r_i^2}{\rho(\theta_i)}\right) + \epsilon L\nu^\star(\Theta)\right) + C\nu^\star(\Theta)\epsilon.
\end{align*}
Let us fix $\epsilon>0$ by minimizing $C\nu^\star(\Theta)\epsilon -\nu^\star(\Theta)d\log(\epsilon)/\tau$, which gives $\epsilon = d/(C\tau)$. The first claim follows by plugging this value for $\epsilon$ in the expression above.

For the second claim of the statement, let us build a suitable candidate $\hat \nu_\epsilon$ in order to upper bound the infimum that defines $\Qq_{\nu^\star,\hat \nu_0}(\tau)$. Let $T$ be an optimal transport map from $\nu_0$ to $\hat \nu_0$ for $W_\infty$, i.e.\ a measurable map $T:\Theta\to\Theta$ satisfying $T_\# \nu_0 = \hat \nu_0$ and $\max\{\dist(\theta,T(\theta))\;;\; \theta\in \spt \nu_0(=\Theta)\} = W_\infty(\nu_0,\hat \nu_0)$ (see~\cite[Sec. 3.2]{santambrogio2015optimal}, the absolute continuity of $\nu_0$ is sufficient for such a map to exist). Now we define $\hat \nu_\epsilon = T_\# \nu_\epsilon$ where $\nu_\epsilon$ is such that $\Hh(\nu_\epsilon,\nu_0)<\infty$. Since the relative entropy is non-increasing under pushforwards, it holds $\Hh(\hat \nu_\epsilon,\hat \nu_0)\leq \Hh(\nu_\epsilon,\nu_0)$. Moreover, it holds $\Vert \nu_\epsilon -\hat \nu_\epsilon\Vert_{\BL}^*\leq W_1(\nu_\epsilon,\hat \nu_\epsilon)\leq \nu^\star(\Theta) W_\infty(\nu_\epsilon,\hat \nu_\epsilon)$. Thus we have
\begin{align*}
\Qq_{\nu^\star,\hat \nu_0}(\tau) &\leq \Vert \nu^\star - \hat \nu_\epsilon\Vert_{\BL}^* +\frac1\tau \Hh(\hat \nu_\epsilon,\hat \nu_0) \\
&\leq \Vert \nu_\epsilon -\hat \nu_\epsilon\Vert_{\BL}+ \Vert \nu^\star - \nu_\epsilon\Vert_{\BL} +\frac1\tau \Hh(\nu_\epsilon,\nu_0)\\
&\leq \nu^\star(\Theta) W_\infty(\nu_\epsilon,\hat \nu_\epsilon) + \Vert \nu^\star - \nu_\epsilon\Vert_{\BL}^* +\frac1\tau \Hh(\nu_\epsilon,\nu_0) .
\end{align*}
The claim follows by noticing that,  by construction, $W_\infty(\nu_\epsilon,\hat \nu_\epsilon) \leq W_\infty(\nu_0,\hat \nu_0)$ and then by taking the infimum in $\nu_\epsilon$.
\end{proof}

\section{Global convergence for gradient descent}\label{app:gradientdescent}

In the following, result, we study the non-convex gradient descent updates $\mu_{k+1} = (T_k)_\# \mu_k$ and $\nu_k=\h\mu_k$ where
\[
T_k(r,\theta) = \Ret_{(r,\theta)}(-2 \alpha_k J'_{\nu_k}(\theta), -\beta_k \nabla J'_{\nu_k}(\theta))
\]
with step-sizes $\alpha, \beta>0$. When $\beta=0$, we recover mirror descent updates in $\Mm_+(\Theta)$ with the entropy mirror map (more specifically, this is true when $\Ret$ is the ``mirror'' retraction defined in Section~\ref{sec:gradientdescent}).

\begin{lemma}\label{lem:discretemirrordecrease}
Assume ${\sf (A1-3)}$ and that $J$ admits a minimizer $\nu^\star\in \Mm_+(\Theta)$. Then there exists $C,\eta_{\max}>0$  such that for all $\nu_0\in \Mm_+(\Theta)$, denoting $B = \sup_{J(\nu)\leq J(\nu_0)} \Vert J'_{\nu}\Vert_{\BL}$, if $\max\{\alpha,\beta\}<\beta_{\max}$, it holds
\begin{align*}
J(\nu_k)-J^\star \leq \inf_{k' \in [0,k]}\left( B\cdot \Qq_{\nu^\star,\nu_0}(4B \alpha k ) +  C\alpha + \beta B^2k \right).
\end{align*}
\end{lemma}
\begin{proof}
As in the proof of Lemma~\ref{lem:sufficientdescrease}, we define $(T_k^r(\theta),T_k^\theta(\theta)) \coloneqq T_k(1,\theta)$ and we define recursively $\nu^\epsilon_{k+1}=(T^\theta_k)_\# \nu^\epsilon_k$ where $\nu^\epsilon_0$ is such that $\Hh(\nu^\epsilon_0,\nu_0)<\infty$. Using the invariance of the relative entropy under diffeomorphisms (indeed, $T^\theta_k$ is a diffeomorphism of $\Theta$ for $\beta$ small enough), and doing a first order expansion of $T^r_k = 1 -2\alpha J'_{\nu_k} +O(\alpha^2)$ it holds for $\beta$ small enough
\begin{align*}
\frac1{4\alpha} \left(\Hh(\nu^\epsilon_{k+1},\nu_{k+1})-\Hh(\nu^\epsilon_{k},\nu_k)\right) & =\frac1{4\alpha} (\Hh(\nu^\epsilon_k,(T^r_k)^2\nu_k)-\Hh(\nu^\epsilon_k,\nu_k)) \\
& = \frac{1}{4\alpha}\Big( \int \log((T_k^r)^{-2})\d\nu_k^\epsilon +\int ((T_k^r)^2-1)\d\nu_k \Big) \\
&= \int J'_{\nu_k}\cdot \d(\nu^\epsilon_k-\nu_k) + O(\alpha)\\
&=\int J'_{\nu_k}\d(\nu^\star -\nu_k) + \int J'_{\nu_k}\d (\nu^\epsilon_k - \nu^\star) + O(\alpha)\\
&\leq -(J(\nu_k)-J^\star)+\Vert J'_{\nu_k}\Vert_{\BL}\cdot \Vert\nu^\star-\nu^\epsilon_k\Vert_{\BL}^* + C\alpha 
\end{align*}
where the term in $O(\alpha)$ originates from a first order approximation of the retraction. Now, taking $\max\{\alpha,\beta\}$ small enough to ensure decrease of $(J(\nu_k))_k$ (by Lemma~\ref{lem:sufficientdescrease}) so that $C$ above can be chosen independently of $k$, it follows
\begin{align*}
\left(\frac1k \sum_{k'=0}^{k-1} J(\nu_{k'})\right) - J^\star 
&\leq  \frac1{4\alpha k}\Hh(\nu^\epsilon_0,\nu_0) + B\Vert \nu^\star - \nu^\epsilon_0 \Vert_{\BL}^* + \frac{B}{k}\left( \sum_{k'=0}^{k-1} \Vert \nu_{k'}^\epsilon -\nu_0^\epsilon\Vert \right)+ C\alpha \\
&\leq B \Qq_{\nu^\star,\nu_0}(4B \alpha k ) +  C\alpha + \frac12 {\beta B^2(k-1)}
\end{align*}
by bounding each term $\Vert \nu_{k'}^\epsilon -\nu_0^\epsilon\Vert$ by $B\beta k'$.
\end{proof}

\begin{proof}[Proof of Theorem~\ref{th:convergencedescent} (gradient descent)]
The proof follows closely that of Theorem~\ref{thm:convergenceflow} but we do not track the ``constants'' (this would be more tedious). By Lemma~\ref{lem:qdensity}, there exists $C>0$ (that depends on $\bar \Hh$, the curvature of $\Theta$ and $\nu^\star(\Theta)$) such that $\Qq_{\nu^\star,\hat \nu_0}(\tau)\leq C(\log \tau)/\tau +\nu^\star(\Theta)W_\infty(\nu_0,\hat \nu_0)$. Combining this with Lemma~\ref{lem:discretemirrordecrease}, we get that when $\max\{\alpha,\beta\} \leq \eta_{\max}$,
\[
J(\nu_k)-J^\star \leq C\frac{\log(B \alpha k)}{\alpha k}  + \beta B^2k+ C'\alpha + B \nu^\star(\Theta)\cdot W_\infty(\nu_0,\hat \nu_0).
\]
Our goal is to choose $k_0,\alpha, \beta$ and $W_\infty(\nu_0,\hat \nu_0)$ so that this is quantity smaller than $\Delta_0\coloneqq J_0-J^\star$. With $\alpha=1/\sqrt{k}$ and $\beta = \beta_0/k$ we get
\[
J(\nu_k)-J^\star \leq \frac{C'\log(Bk)}{\sqrt{k}} + B^2\beta_0+ B \nu^\star(\Theta) \cdot W_\infty(\nu_0,\hat \nu_0).
\]
Then, using a bound $\log(u)\leq C_\epsilon u^\epsilon$, we may choose $k \gtrsim \Delta_0^{-2-\epsilon}$, $\beta_0 \leq\frac13 \Delta_0/B^2$ and $W_\infty(\nu_0,\hat \nu_0)\leq\frac13 \Delta_0/(B\nu^\star(\Theta))$ in order to have $J(\nu_k)-J^\star \leq \Delta_0$. This gives $\alpha \lesssim \Delta_0^{1+\epsilon/2}$, $\beta \lesssim \Delta_0^{3+\epsilon}$ and the regime of exponential convergence kicks off after $k=\Delta_0^{-2-\epsilon}$ iterations.
\end{proof}

\section{Faster rate for mirror descent}\label{app:fasterrate}
In this section, we show that for a specific choice of retraction, the convergence rate of $O(\log(t)/t)$ for the gradient flow is preserved for the gradient descent.

\begin{proposition}[Mirror flow, fast rate]\label{prop:mirrorflow}
Assume {\sf (A1-4)} and consider the infinite dimensional mirror descent update 
\[
\nu_{k+1}=\exp(-4\alpha J'_{\nu_k})\nu_{k}
\]
which corresponds to the so-called mirror retraction in Section~\ref{sec:gradientdescent} and $\beta=0$. For any $\nu_0 \in \Mm_+(\Theta)$, there exists $\alpha_{\max}>0$ such that for $\alpha\leq \alpha_{\max}$ it holds, denoting $B_{\nu_0} = \sup_{J(\nu)\leq J(\nu_0)} \Vert J'_{\nu}\Vert_{\BL}$,
\[
J(\nu_k) - J^\star \leq B_{\nu_0}  \Qq_{\nu^\star,\nu_0}(2\alpha B_{\nu_0}  k).
\]
In particular, combining with Lemma~\ref{lem:qdensity}, if $\nu_0 = \rho\!\vol$ has a smooth positive density, then  $J(\nu_k)-J^\star = O(\log(k)/k)$. 
\end{proposition}
\begin{proof}
Consider $\nu_\epsilon \in \Mm_+(\Theta)$ such that $\Hh(\nu_\epsilon,\nu_0)<\infty$. It holds
\begin{align*}
\frac1{4\alpha}\left( \Hh(\nu_\epsilon,\nu_k)-\Hh(\nu_\epsilon,\nu_{t+1})\right) &= -  \frac1{4\alpha} \int \log\left( \frac{\nu_{k+1}}{\nu_k}\right)\d(\nu_{k+1}-\nu_\epsilon)+\frac{1}{4\alpha} \Hh(\nu_{k+1},\nu_k)\\
&=\int J'_{\nu_k}\d(\nu_{k+1}-\nu_\epsilon)+\frac{1}{4\alpha} \Hh(\nu_{k+1},\nu_k)
\end{align*}
where the first equality is obtained by rearranging terms in the definition of $\Hh$, and the second one is specific to the mirror retraction.  Let us estimate the two terms in the right-hand side. Using convexity inequalities, we get
\begin{align*}
\int J'_{\nu_k}\d(\nu_{k+1}-\nu_\epsilon) &=  \int J'_{\nu_k}\d(\nu_{k}-\nu^\star) + \int J'_{\nu_k}\d(\nu_{k+1}-\nu_k)+ \int J'_{\nu_k}\d(\nu^\star-\nu_\epsilon) \\
& \geq  J(\nu_k)-J(\nu^\star)+J(\nu_{k+1})-J(\nu_k) + O(\alpha \Vert g_{\nu_k}\Vert^2_{L^2(\nu_k)})+ \int J'_{\nu_k}\d(\nu^\star-\nu_\epsilon) \\
& \geq J(\nu_{k+1})-J^\star + O(\alpha \Vert g_{\nu_k}\Vert^2_{L^2(\nu_k)}) - \Vert J'_{\nu_k}\Vert_\BL \cdot \Vert \nu^\star-\nu_\epsilon\Vert_\BL^*.
\end{align*}
Here the term in $O(\alpha \Vert g_{\nu_k}\Vert^2_{L^2(\nu_k)})$ comes from the proof of Lemma~\ref{lem:sufficientdescrease} (note that the iterates remain in a sublevel of $J$ for $\alpha$ small enough). As for the relative entropy term, we have, using the convexity inequality $\exp(u)\geq 1+ u$,
\begin{align*}
\frac1{\alpha}\Hh(\nu_{k+1},\nu_k) &= \frac1\alpha \int  \left( \exp(-2\alpha J'_{\nu_k}) (-2\alpha J'_{\nu_k}- 1) + 1 \right)\d\nu_{k}\\
&\geq  \frac1\alpha \int \left(4\alpha^2 \vert J'_{\nu_k}\Vert^2 -1 +1\right)\d\nu_k =   \int 4\alpha\vert J'_{\nu_k}\vert^2\d\nu_k = \Vert g_{\nu_k}\Vert^2_{L^2(\nu_k)}.
\end{align*}
We use this inequality in place of the strong convexity of the mirror function used in the usual proof of mirror descent (because there is no Pinsker inequality on $\Mm_+(\Theta)$). Coming back to the first equality we have derived, it holds,
\begin{align*}
B_{\nu_0} \Vert \nu^\star-\nu_\epsilon\Vert_\BL^*+ \frac1{4\alpha}\left( \Hh(\nu_\epsilon,\nu_k)-\Hh(\nu_\epsilon,\nu_{t+1})\right) &\geq   J(\nu_{k+1})-J^\star + \frac14 \Vert g_{\nu_k}\Vert^2_{L^2(\nu_k)} + O(\alpha \Vert g_{\nu_k}\Vert^2_{L^2(\nu_k)})
\end{align*}
Thus for $\alpha$ small enough, it holds 
\[
 B_{\nu_0}  \Vert \nu^\star-\nu_\epsilon\Vert_\BL^*+\frac1{4\alpha}\left( \Hh(\nu_\epsilon,\nu_k)-\Hh(\nu_\epsilon,\nu_{t+1})\right) \geq   J(\nu_{k+1})-J^\star.
 \]
 Summing over $K$ iterations and dividing by $K$, we get
\[
\left( \frac1K \sum_{k=1}^{K} J(\nu_{k}) \right) - J^\star \leq \frac1{4\alpha K} \Hh(\nu_\epsilon,\nu_K) + B_{\nu_0} \Vert \nu^\star-\nu_\epsilon\Vert_\BL^*.
\]
Since for $\alpha$ small enough $(J(\nu_k))_{k\geq 1}$ is decreasing  (by Lemma~\ref{lem:sufficientdescrease}), the result follows.
\end{proof}

\section{Convergence rate for lower bounded densities}\label{sec:boundeddensity}
In this section, we justify the claim made in Section~\ref{sec:fullynonconvex} about the convergence without condition on $\beta/\alpha$. Let us recall the result that we want to prove.
\begin{proposition} Under {\sf (A1-3)}, for any $J_{\max}>J^\star$, there exists $C>0$ such that for any $\eta,t>0$ and $\nu_0\in \Mm_+(\Theta)$ satisfying $J(\nu_0)\leq J_{\max}$, if the projected gradient flow~\eqref{eq:projectedGF} satisfies for $0\leq s \leq t$,
\[
\nu_s\vert_{S_t} \geq \eta \vol\vert_{S_t}
\]
where $S_t = \{\theta \in \Theta\;;\; J'_{\nu_s}(\theta)\leq 0 \text{ for some } s\in {[0,t]}\}$, then 
$
J(\nu_t)-J^\star\leq \frac{C}{\sqrt{\alpha \eta t}}.
$
\end{proposition}

\begin{proof}
Following~\cite{wei2019regularization}, we start with the convexity inequality
\begin{align*}
J(\nu_t)-J^\star \leq \int J'_{\nu_t}\d\nu_t - \int J'_{\nu_t}\d\nu^\star .
\end{align*}
Let us control these two terms separately. On the one hand, one has by Jensen's inequality
\begin{align*}
\left(\int J'_{\nu_t}\d\nu_t \right)^2 =(\nu_t(\Theta))^2\left(\frac{1}{\nu_t(\Theta)}\int J'_{\nu_t}\d\nu_t\right)^2 \leq \nu_t(\Theta) \int  \vert J'_{\nu_t}\vert^2 \d\nu_t \leq - \frac{\nu_t(\Theta)}{4\alpha} \frac{\d}{\d t} J(\nu_t).
\end{align*}
Using the fact that on sublevels of $J$, $\nu(\Theta)$ and $\Vert g_{\nu}\Vert^2_{L^2(\nu)}$ are bounded, we have, for some $C>0$,
\[
\int J'_{\nu_t}\d\nu_t\leq C\left(-\frac{1}{\alpha}\frac{\d}{\d t} J(\nu_t)\right)^{1/3}.
\]
On the other hand, we have
\begin{align*}
\int J'_{\nu_t}\d\nu^\star \geq \nu^\star(\Theta)\min \left\{ 0, \min_{\theta\in \Theta} J'_{\nu_t}(\theta)\right\} \eqqcolon \nu^\star(\Theta)\cdot v_t.
\end{align*}
where the last equality defines $v_t\leq 0$. Using the gradient flow structure, let us show that a non-zero $v_t$ and a lower bound $\eta$ on the density of $\nu_t$ (at least on the set $\{J'_{\nu_t}\leq 0$\}) guarantees a decrease of the objective. Indeed, letting $\Theta_t = \{\theta \in \Theta\;;\; J'_{\nu_t}(\theta) \leq v_t /2\}$ (which could be empty), we get
\begin{align*}
- \frac{\d}{\d t} J(\nu_t)\geq 4\alpha \int_{\Theta_t} \vert J'_{\nu_t}\vert^2 \d\nu_t \geq 4\alpha (v_t/2)^2 \nu_t(\Theta_t) = \alpha \cdot v_t^2 \cdot\eta\cdot \vol(\Theta_t).
\end{align*}
Moreover, the Lipschitz regularity of $J'_{\nu}$ is bounded on sublevels of $J$, and thus along gradient flow trajectories, so there exists $C'>0$ such that $\vol(\Theta_t)\geq C'\cdot \vert v_t\vert$. It follows
\[
\vert v_t\vert^3 \leq -\frac{1}{C'\alpha\eta}\frac{\d}{\d t}J(\nu_t) \quad \Rightarrow \quad v_t \geq -\left(-\frac{1}{C'\alpha\eta}\frac{\d}{\d t}J(\nu_t)\right)^{1/3}.
\]
Coming back to our first inequality, we have
\[
J(\nu_t)-J^\star \leq C\left(-\frac{1}{\alpha}\frac{\d}{\d t} J(\nu_t)\right)^{1/3} + \left(-\frac{1}{C'\alpha\eta}\frac{\d}{\d t}J(\nu_t)\right)^{1/3} \leq \frac{C''}{(\alpha\eta)^{1/3}}\left(-\frac{\d}{\d t}J(\nu_t)\right)^{1/3}
\]
for some $C''>0$ that,  given $J(\nu_0)$, is independent of $\alpha,\eta$ and $\nu_t$. It remains to remark that a continuously differentiable and positive function $h$ that satisfies $h(t)\leq C^{-1/3} \cdot (-h'(t))^{1/3}$ satisfies $C\leq -h'(t)/h(t)^3 = \frac12  \frac{\d}{\d t}(h(t)^{-2})$ and, after integrating between $0$ and $t$, $h(t)\leq \left(2Ct +h(0)^{-2}\right)^{-1/2}\leq \frac{1}{\sqrt{2Ct}}$. We conclude by taking $h(t)=J(\nu_t)-J^\star$ and $C\propto \alpha \eta$.
\end{proof}
\end{document}